\documentclass[12pt, reqno]{amsart}
     \makeatletter
     \def\section{\@startsection{section}{1}%
     \z@{.7\linespacing\@plus\linespacing}{.5\linespacing}%
     {\bfseries
     \centering
     }}
     \def\@secnumfont{\bfseries}
     \makeatother
\usepackage{amssymb}
\usepackage[final]{hyperref}

\usepackage{amsrefs}

\usepackage{tikz}

\usepackage{tikz-3dplot}

\usepackage{pgfplots}

\pgfplotsset{%
    compat=1.8,
    compat/show suggested version=false,
}

\usetikzlibrary{calc,fadings,decorations.pathreplacing}

\usepackage[many]{tcolorbox}

  \usepackage{wrapfig}

\usetikzlibrary{fadings}
\pgfdeclareverticalshading{up}{100bp}
{color(0bp)=(pgftransparent!0); color(50bp)=(pgftransparent!0);
 color(75bp)=(pgftransparent!25); color(100bp)=(pgftransparent!25)}
\pgfdeclarefading{up}{\pgfuseshading{up}}    
\pgfdeclareradialshading{thick ring}{\pgfpointorigin}{
  color(0pt)=(pgftransparent!100); color(15pt)=(pgftransparent!100); 
  color(20)=(pgftransparent!0);
  color(25bp)=(pgftransparent!100); color(50bp)=(pgftransparent!100)}
\pgfdeclarefading{thick ring}{\pgfuseshading{thick ring}}    
\begin{tikzfadingfrompicture}[name=ball]
\shade [shading=ball, ball color=black] circle [radius=1];
\end{tikzfadingfrompicture} 
\begin{tikzfadingfrompicture}[name=ring]
\fill [white, path fading=thick ring] circle [radius=1];
\fill [black, path fading=up, fading angle=215] circle [radius=1];
\end{tikzfadingfrompicture}    

\tikzset{shiny ball/.style={
  fill=none, draw=none, shading=ball, shading angle=-15,
  postaction={fill=white, path fading=ball, opacity=0.75, fading angle=45},
  postaction={fill=white, path fading=ring}
}}

  \tikzstyle{mybox} = [draw=black, fill=black!20, very thick,
    rectangle, rounded corners, inner sep=10pt, inner ysep=20pt]
\tikzstyle{fancytitle} =[fill=black, text=white]

\newcommand\pgfmathsinandcos[3]{%
  \pgfmathsetmacro#1{sin(#3)}%
  \pgfmathsetmacro#2{cos(#3)}%
}

\newcommand\LatitudePlane[3][current plane]{%
  \pgfmathsinandcos\sinEl\cosEl{#2} 
  \pgfmathsinandcos\sint\cost{#3} 
  \pgfmathsetmacro\yshift{\cosEl*\sint}
  \tikzset{#1/.estyle={cm={\cost,0,0,\cost*\sinEl,(0,\yshift)}}} %
}

\newcommand\DrawLatitudeCircle[2][1]{
  \LatitudePlane{\angEl}{#2}
  \tikzset{current plane/.prefix style={scale=#1}}
  \pgfmathsetmacro\sinVis{sin(#2)/cos(#2)*sin(\angEl)/cos(\angEl)}
  \pgfmathsetmacro\angVis{asin(min(1,max(\sinVis,-1)))}
  \draw[current plane] (\angVis:1) arc (\angVis:-\angVis-180:1);
  \draw[current plane,dashed] (180-\angVis:1) arc (180-\angVis:\angVis:1);
}

\tikzset{%
  >=latex, 
  inner sep=0pt,%
  outer sep=2pt,%
  mark coordinate/.style={inner sep=0pt,outer sep=0pt,minimum size=3pt,
    fill=black,circle}%
}

 \usepackage{wrapfig}

\usepackage{enumerate}

\newtheorem{theorem}{Theorem}[section]
\newtheorem{lemma}{Lemma}[section]
\newtheorem{prop}{Proposition}[section]

\theoremstyle{definition}

\numberwithin{equation}{section}

\newcommand{\End}{ {\rm End}}

\newcommand{\mbr}{{\mathbb R}}

\newcommand{\cl}{{\mathcal L}}

\newcommand{\ovs}{\overline{\sigma}}
\newcommand{\ovm}{\overline{\mu}}

\newcommand{\la}{{\langle}}

\newcommand{\ra}{{\rangle}}

\newcommand{\norm}[1]{\left\lVert#1\right\rVert}

\begin{document}

\title[The Gaussian Limit for Spherical Means]{The Gaussian Limit for High-Dimensional  Spherical Means}



\author{Amy Peterson}
\address{Department of Mathematics \\
  University  of Connecticut\\
Storrs, CT 062569 \\
e-mail: \sl amy.peterson@uconn.edu}
 
\author{Ambar N.~Sengupta}
\address{Department of Mathematics \\
  University  of Connecticut\\
Storrs, CT 062569\\
e-mail: \sl ambarnsg@gmail.com}
\thanks{Research   supported by NSA grant H98230-15-1-0254}


\date{9 May 2018}

\dedicatory{}

\begin{abstract} We show that the limit of integrals along slices of a high dimensional sphere is a Gaussian integral on a corresponding finite-codimension affine subspace in infinite dimensions. 
\end{abstract}

\maketitle
 
\section{Introduction}\label{S:Intro}

 The Radon transform \cite{ Radon17} of a function $f$ on $\mbr^N$ associates to each affine subspace $L$ in $\mbr^N$ the (Lebesgue) integral of $f$ over $L$. In the infinite-dimensional setting, the Gaussian Radon transform \cite{HolSenGR2012} of a function $\phi$ on a Banach space $B$ associates to each closed affine subspace $L$ a Gaussian integral of $f$ over $L$. Returning to the finite-dimensional setting again, we can associate to a function $f$, defined on an affine subspace $L$ of $\mbr^N$, the integral of $f$ over the `circle' in which $L$ intersects the sphere $S^{N-1}(\sqrt{N})$ of radius $\sqrt{N}$ in $\mbr^N$. Our goal in this paper is to show that these spherical integrals yield the infinite-dimensional Gaussian Radon transform in the large-$N$ limit,   when $L$ is of finite codimension.  The  case of hyperplanes was established in \cite{SenGRL2016}. More background is provided below in subsection \ref{ss:rl}. The present paper may be viewed as the third in a series, following \cite{HolSenGR2012}  and  \cite{SenGRL2016}, but can be read independently of the earlier papers.

 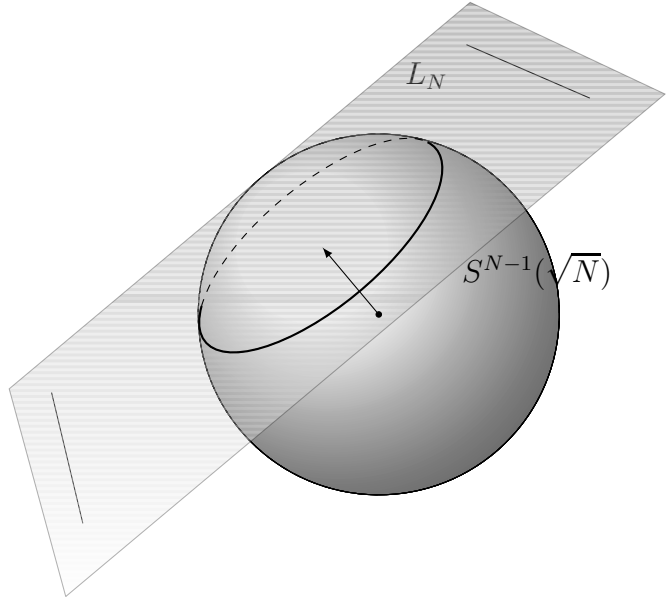
\begin{figure}
\tikzset{
    partial ellipse/.style args={#1:#2:#3}{
        insert path={+ (#1:#3) arc (#1:#2:#3)}
 }}

 {
\centering

   \begin{tikzpicture} 
[scale=.8,   rotate= 40]   
\def\R{3} 
\def\angEl{40} 
\filldraw[ball color=white] (0,0) circle (\R);
 \foreach \t in {-80,-40,...,80} { \DrawLatitudeCircle[\R]{\t} }

\coordinate [label=left:\textcolor{black}{${}$}](G) at (-6,.5);

\coordinate [label=left:\textcolor{black}{${}$}] (C) at (5,.5);

\coordinate (F) at (4,2.5);

\coordinate (H) at (-5 ,2.5);


 \coordinate [label=left:\textcolor{black}{$L_N$}] (L) at (3.5, 2.25);
 \coordinate [label=left:\textcolor{black}{$S^{N-1}(\sqrt{N})$}] (X) at 
 (3.5, -2);

 \draw (G)-- (H);
 
 \draw (C)--(F);


\shadedraw [color=black, opacity = .3]  (-7,-.25)--(6,-.25)--(4.5,3)--(-5.5,3)--(-7,-.25);

 \draw[->, shorten <=1pt,shorten >=1pt](0,0)--(0, 1.5);
 
\draw[black,fill=black] (0,0) circle (.25ex); 
 
 \draw [dashed] (0,1.5) ellipse (2.5  and 1);
 
 \draw[thick] (0,1.5) [partial ellipse=150:370: 2.5 and 1];

\end{tikzpicture}

}
\caption{The plane $L_N$ slicing the sphere $S^{N-1}(\sqrt{N})$}
\label{F:slice}
\end{figure}

 \subsection{Summary description of results} Let  $l^2$ be the subspace of the space $\mbr^\infty$ of all real sequences   $(x_n)_{n\geq 1}$ for which the standard $l^2$ norm $\bigl (\sum_n x_n^2\bigr)^{1/2}$ is finite.  Let $L$ be an affine subspace of $l^2$ of finite codimension, and $L_N$  the subspace of $\mbr^N$ consisting of all points $(x_1,\ldots, x_N)$ such that $(x_1,\ldots, x_N,0,0,\ldots)\in L$. 
 Then  the affine subspace $L_N$  intersects the sphere $S^{N-1}(\sqrt{N})$, centered at $0$ and having radius $\sqrt{N}$, in a `circle' $S_{L_N}$. Let  $  \sigma$ denote the standard surface measure on any sphere, and let  $  \ovs$   denote the measure $\sigma$ normalized to have unit total mass.

Let $f$ be a bounded Borel function on $\mbr^k$, so that we have a corresponding function on $\mbr^N$, with $N>k$, whose value at any $(x_1,\ldots, x_N)$ is $f(x_1,\ldots, x_k)$.
 
 We prove (Theorem \ref{T:RadnonNlim})  the following limiting formula:
\begin{equation}\label{E:limGR}
\lim_{N\to\infty}\int_{S_{L_N}} f(x_1,\ldots, x_k)\, d\ovs(x)=\int_{\mbr^\infty}f(x_1,\ldots, x_k)\,d\ovm_{L}(x)
\end{equation}
where the integration on the right in (\ref{E:limGR}) is with respect to a probability measure $\ovm_L$ over $\mbr^\infty$, the space of all sequences $(x_n)_{n\geq 1}$ of real numbers.   
The   probability measure $\ovm_{L}$ on $\mbr^\infty$  is uniquely specified by the characteristic function
 \begin{equation}\label{E:cfhyperpl1}
\int_{\mbr^\infty}\exp\left(i\sum_{j=1}^\infty t_jx_j\right)\,d\ovm_{L}(x)=\mbox{exp}\left({i  \la t, p_L\ra_{l^2} -\frac{|P_{L}t|_{l^2}^2}{2}}\right),
\end{equation}
for all $t\in\mbr^{\infty}$ for which all but finitely many components are $0$, the point $p_L\in l^2$ is the point on $L$ closest to the origin, and $P_L$ is the orthogonal projection operator in $l^2$ onto the subspace $L-p_L$. Our main result, Theorem \ref{T:mainAWS},  is a formulation of the limit formula (\ref{E:limGR})   in Banach spaces.
   
\subsection{Related literarure}\label{ss:rl}  The immediate predecessor for our work is \cite{SenGRL2016}, where the corresponding result was proved for the case where $L$ has codimension $1$.     

The connection  between Gaussian measure and the uniform measure on high dimensional spheres appeared originally in the works of Maxwell  \cite{MaxGas1860}  and Boltzmann  \cite[pages 549-553]{BoltzStud1868}. Later works included Wiener's paper \cite{WienDS1923} on ``differential space'', L\'evy \cite{LevyPLAF1922}, McKean \cite{McKGDS1973}, and Hida \cite{HidaStat1970}.   

Measures  on infinite-dimensional manifolds have been studied in many works, such as Skorohod \cite{SkorSurf1970}, Uglanov \cite{UglaSurf2000}, da Prato et al. \cite{DaPraSurf2014, DaPraSurf2017}, Feyel et al. \cite{FeyelHMW1992, FeyelHG2001}, Kuo et al. \cites{Kuo1972, KuoSurf2010}, and Weitsman \cite{Weits2008, Weits2010} in the context of quantum field theory. For  the theory of Gaussian measures in infinite dimensions we refer to the monographs of Bogachev \cite{BogGM1998} and Kuo \cite{KuoGB1975}.  Hertle \cite{HertSurf1979, HertGa1982}  defined surface measures on spheres and hyperplanes in infinite dimensions by a   method different from the one we use and studied the Radon transform using these surface measures.

The approach to Gaussian measures on hyperplanes, and more generally affine subspaces, in infinite dimensions that we use was initiated in  \cite{MiSen2007}, where such measures where defined  for hyperplanes in Hilbert spaces and an inversion formula obtained for the Gaussian Radon transform. In \cite{BecSenSup2012} it was shown that if a suitably well-behaved function, on an infinite-dimensional Hilbert space, has zero Gaussian integrals on hyperplanes not intersecting a closed, bounded, convex set then the function is supported within this set; this is an infinite-dimensional counterpart of Helgason's support theorem in Radon transform theory  \cite{Helg1999}. A support theorem in the setting of white noise analysis was proved by Becnel \cite{BecSup2015} using a different strategy.   The Gaussian Radon transform was developed for Banach spaces in \cite{HolSenGR2012}, where again a support theorem was established.    Bogachev and of Lukintsova \cites{BogLukRT2012, LukRad2013} studied the Radon transform of more general Radon measures in infinite dimensions and established results on the support behavior of the Radon transform.  

There is a vast literature on the subject of finite-dimensional Radon transforms; we refer to Rubin \cite{Rubin2004, Rubin2015} for Radon transforms on Grassmannians and for a broader overview of the subject.

\section{The Gaussian Limit of Spherical Radon Transforms}\label{s:splim}

In this section $L$ denotes an affine subspace  in $l^2$ of finite codimension. Thus $L$ is of the form
\begin{equation}\label{E:splimL}
L=\{v\in l^2: \la v,w_1\ra=p_1,\ldots, \la v, w_m\ra=p_m\},
\end{equation}
where  $w_1,\ldots, w_m\in l^2$ are orthonormal and $p_1,\ldots, p_m\in\mbr$.

Let us view $\mbr^N$ as a subspace of $l^2$ by identifying any vector $z\in\mbr^N$ with the vector
$$(z, 0,0,0,\ldots)\in l^2.$$
Let
\begin{equation}\label{E:LN}
L_N=L\cap \mbr^N.
\end{equation}
The definition (\ref{E:splimL})  of $L$ shows that a vector $z\in\mbr^N$ lies in $L_N$ if and only if 
\begin{equation}\label{E:zwi}
\la z, (w_1)_{(N)}\ra=p_1,\ldots, \la z, (w_m)_{(N)}\ra=p_m,
\end{equation}
where, for any sequence $x=(x_1, x_2,\ldots)$,
$$x_{(N)}=(x_1,\ldots, x_N).$$
We work with $N$ large enough, larger than some $N_0$, so that $(w_1)_{(N)},\ldots,  (w_m)_{(N)}$ are linearly independent vectors; this ensures also that $L_N\neq\emptyset$ when $N\geq N_0$.
If $z, z'\in L_N$ then $z-z'$ is orthogonal to the vectors $(w_1)_{(N)}, \ldots, (w_m)_{(N)}$. 
On the other hand if $z'\in L_N$ and if $z\in\mbr^N$ is such that $z-z'$ is orthogonal to $(w_1)_{(N)}, \ldots, (w_m)_{(N)}$ then by (\ref{E:zwi}) $z\in L_N$.
Let 
$$p^0_N=\hbox{point of $L_N$ closest to $0$.}$$
Thus,
\begin{equation}\label{E:splimLN}
L_N=p^0_N+[(w_1)_{(N)},\ldots, (w_m)_{(N)}]^\perp\,
\end{equation}
where   
$$[(w_1)_{(N)},\ldots, (w_m)_{(N)}]$$
is the span of the vectors $(w_1)_{(N)},\ldots, (w_m)_{(N)}$.
Let $S_{L_N}$ be the `circle' of intersection of $L_N$ with the sphere
\begin{equation}\label{E:SNmin1}
S^{N-1}(\sqrt{N})
\end{equation}
of radius $\sqrt{N}$ and center $0$ in $\mbr^N$: 
\begin{equation}\label{E:SLN}
S_{L_N}=L_N\cap S^{N-1}(\sqrt{N}).
\end{equation} 
Identifying $\mbr^N$ with the subspace $\mbr^N\times\{0\}\subset l^2$, for each $N$,  we have $L_m\subset L_n$ if $m<n$ and so the distance of $L_n$ from $0$ is at most equal to the distance of $L_m$ from $0$:
\begin{equation}
\norm{p^0_n}\leq \norm{p^0_m}.
\end{equation}

  We  work with  $N$ larger than $N_0$ that also satisfies  
\begin{equation}\label{E:uNgrp}
 \sqrt{N}> \norm{p^0_{N_0}},
\end{equation}
which implies that 
\begin{equation}\label{E:uNgrp2}
 \sqrt{N}>\norm{p^0_N}.
\end{equation}
Thus
    ${S_{L_N}}$ is   a nonempty $(N-m)$-dimensional sphere. The radius of this sphere is
 \begin{equation}\label{E:radSLN}
 {\rm radius}(S_{L_N}) =\sqrt{N-\norm{p^0_N}^2}.
 \end{equation}

One of our goals is to prove the following result:

 \begin{theorem}\label{T:RadnonNlim} Let $L$ be a finite-codimension affine subspace in $l^2$, specified by (\ref{E:splimL}). Let $k$ be a positive integer; suppose that the image of $L$ under the coordinate projection $l^2\to\mbr^k:z\mapsto z_{(k)}=(z_1,\ldots, z_k)$ is all of $\mbr^k$.  Let $\phi$ be a bounded Borel function on $\mbr^k$. Then
 \begin{equation}\label{E:limRfS}
 \lim_{N\to\infty}\int_{S_{L_N}} \phi(x_1,\ldots, x_k)\,d\ovs(x_1,\ldots, x_N)=\int_{\mbr^\infty}\phi(z_{(k)})  \,d\mu(z),  \end{equation}
 where $\ovs$ is the standard surface area measure on $S_{L_N}$  (defined in (\ref{E:SLN})) normalized to unit total  mass, and $\mu $ is the probability measure on $\mbr^\infty$ specified by the characteristic function
  \begin{equation}\label{E:cphi}
  \begin{split}
\int_{\mbr^\infty} \exp\left({i\la t, x\ra}\right)\,d\mu (x) &=\exp\left({i\la t, z^0  \ra-\frac{1}{2}\norm{P_0t}^2}\right)\qquad   \hbox{for all $t\in\mbr^\infty_0$,}
\end{split}
\end{equation}  
where $z^0$ is the point on $L$ closest to the origin and $P_0$ is the orthogonal projection in $l^2$ onto the subspace $L-z^0$. \end{theorem}

 On the left side of (\ref{E:limRfS}) $\phi$ is only evaluated on the image   $\pi_{(k)}(L)$, which is why the assumption that this is all of $\mbr^k$ is relevant.  
 
We will also prove a version (Theorem \ref{T:mainAWS}) of this result in the setting of Hilbert and Banach spaces.
 
 \section{Spherical Disintegration}\label{s:spherdis}
 
The key tool in the proof of Theorem  \ref{T:RadnonNlim}  will be a spherical disintegration formula that we establish in this section in Theorem \ref{T:slicedisint}.

The uniform surface measure on a sphere  can be constructed in several ways. Perhaps the most elementary way is to define it by the traditional Euclidean formula for volume of a cone:
 \begin{equation}\label{E:defsig}
 \sigma(E)=\frac{d+1}{a}\lambda(C_E),
 \end{equation}
 where $E$ is any Borel subset of a $d$-dimensional sphere of radius $a$ and $\lambda(C_E)$ is the volume of the cone $C_E$ with base $E$ and vertex the center of the sphere.  We will use $\sigma$ to denote the uniform measure, defined as above, on any sphere in any dimension.
 
The definition (\ref{E:defsig}) and the scaling property of Lebesgue measure in $\mbr^{d+1}$ leads to the scaling formula:
 \begin{equation}\label{E:spherscal}
  \int_{S^d(r)}f\,d\sigma =(r/a)^d\int_{S^d(a)} f\bigl((r/a)z\bigr)\, d\sigma(z),
 \end{equation}
 whenever either side exists, where $S^d(t)$ denotes the sphere of radius $t$ and center $0$ in $\mbr^{d+1}$. The polar disintegration formula 
 \begin{equation}\label{E:polard}
 \int_{\mbr^{d+1}}f\,d x=\int_{r\in (0,\infty)}\left[\int_{S^d(r)}f\,d\sigma\right]\,dr,
 \end{equation}
holds as well.  Elementary proofs of these formulas are in \cite{SenGRL2016}.

\subsection{Disintegration by slices} The following is sightly more general phrasing of a geometric disintegration formula from \cite{SenGRL2016}:

\begin{theorem}\label{T:spherdisint} Let $f$ be a  non-negative or bounded Borel function on the sphere $S_V(a)$ of radius $a>0$ and center $0$ in a finite-dimensional real inner-product space $V$. Let $W$ be a proper subspace of $V$. Then
\begin{equation}\label{E:spheredisint}
\int_{S_V(a)}f\,d\sigma =\int_{B_W(a)}\left[\int_{y\in S_V(a_{x})\cap W^\perp} f(x+y)\,d\sigma(y)\right]\frac{a}{a_{x}}\,dx,
\end{equation}
where 
\begin{equation}\label{E:defax}
a_w=\sqrt{a^2-\norm{w}^2},
\end{equation}
and $B_W(a)$ is the open ball of radius $a$,  and center $0$, in $W$.
 
\end{theorem}
Note that formula (\ref{E:spherscal}) holds whenever $f$ is integrable on $S_V(a)$ since it holds for  non-negative  $f$.

With notation as above, let
$$P:V\to V$$
be  the  orthogonal projection onto an affine subspace $W$; thus  $Pz$ is the point on $W$ closest to $z$. Then the disintegration formula (\ref{E:spheredisint}) can be expressed as:
\begin{equation}\label{E:disintfsigP}
\int_{S_{V}(a)}f\,d\sigma= \int_{B_W(a)}\left[\int_{S_{V}(a)\cap P^{-1}(x)}f\,d\sigma\right]\frac{a}{\sqrt{a^2-\norm{x}^2}}\,dx.
\end{equation}

Since the expression $\sqrt{a^2-\norm{x}^2}$ will keep appearing we will use the notation from (\ref{E:defax}):
\begin{equation}\label{E:ax}
a_x=\sqrt{a^2-\norm{x}^2}.
\end{equation}
We will use other similar notation, such as $a_t$ when $t\in\mbr$.

 For a point $x$ inside the ball $B_{V}(a)$ of radius $a$ in $V$,  the geometric meaning of $a_x$ is the radius of the slice $S_{V}(a)\cap P^{-1}(x)$. If $z\in S_{V}(a)$ lies on $P^{-1}(x)$ then we can write $z$ as $x$ plus a `radial vector' orthogonal to $x$ from the center $x$ of the slice to $z$:
$$z= Pz +z-Pz =x+ (I-P)z,$$
where $(I-P)z\in \ker P$ is orthogonal to the image ${\rm Im}(P)$ and hence also orthogonal to $x$.  Thus
\begin{equation}\label{E:zx}
\norm{z-x}^2= \norm{z}^2-\norm{x}^2=a_x^2,
\end{equation}
which means that any $z\in S_{V}(a)\cap P^{-1}(x)$ lies at the fixed distance $a_x$ from $x$.

The geometric meaning of $a/a_x$ is $1/\cos\theta_x$, where $\theta_x$ is the angle between the vector from $x$ to $z$ and the vector $z$:
\begin{equation}\label{E:axthet}
\frac{a}{a_x}=\frac{\norm{z}}{\norm{z-x}}=\frac{1}{\cos\theta_x}.
\end{equation}
This is illustrated in Figure \ref{F:angleandtrig}.

 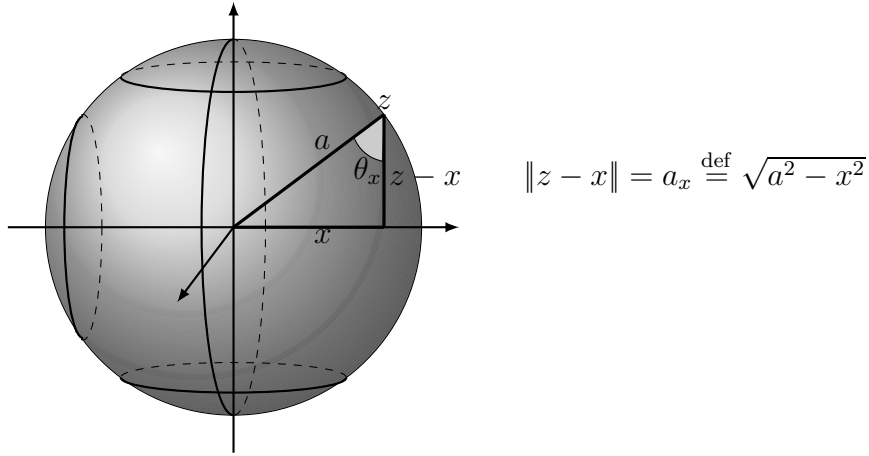
\begin{figure}
\tikzset{
    partial ellipse/.style args={#1:#2:#3}{
        insert path={+ (#1:#3) arc (#1:#2:#3)}
 }}

\centering
 \begin{tikzpicture}[scale=2.5] 
 

\shadedraw[line width=0.15mm, white,  ball color = gray!40, opacity = 0.6 ] (0, 0, 0) circle (1);


\draw (0,0) circle (1);  


\draw[thick, ->] (-1.2,0) -- (1.2,0) coordinate (x axis) node[right] {$ $ };
\draw[thick,->] (0,-1.2) -- (0,1.2) coordinate (z axis)node[above] {$ $ };
\draw[thick, ->] (0,0) -- (-.3,-.4) coordinate (y axis)node[right] {$ $ };

 \draw [dashed] (0,0) ellipse (.17 and 1);
 \draw[thick] (0,0) [partial ellipse=90:270: .17 and 1];

\draw [dashed] (-.8,0) ellipse (.1 and .6);
 \draw[thick] (-.8,0)[partial ellipse=90:270: .1  and .59];

\draw [dashed] (0,-.8) ellipse (.6 and .08);
  \draw[thick] (0,-.8)[partial ellipse=180:360: (.6 and .08];
  
\draw [dashed] (0,.8) ellipse (.6 and .08);
  \draw[thick] (0,.8)[partial ellipse=180:360: (.6 and .08];

 
\coordinate [label=above:$z$] (P) at (.8,.6);

\coordinate [label=right:$x$] (x) at (.4,-.05);

\coordinate [label=right:$a$] (a) at (.4,.45);

\coordinate [label=right:$\theta_x$] (th) at (.61,.3);

\filldraw[fill=black!20,draw=black!50!black] (.8,.6) -- (.64,.48) arc
(200:260: .2) --(P);

\draw[very thick,black] (0,0)--(.8,0);

\draw[very thick,black] (.8,0) --(0.8,0.3)  node[right ]{$z-x\qquad |\!|z-x|\!|=a_x\stackrel{\rm def}{=}\sqrt{a^2-x^2}$}--(P);


\draw [very thick, black] (0,0)--(P);


\end{tikzpicture}
\caption{Illustration for $a_x$ and $\theta_x$}
\label{F:angleandtrig}
\end{figure}

Formula  (\ref{E:spherscal})  can be proved using polar coordinates or other more differential geometric methods but we present an entirely self-contained  argument. 
 
 \begin{proof}  We assume that $f\geq 0$; all other cases follow by taking real and imaginary parts if $f$ is complex-valued, and positive and negative parts for real-valued $f$.  By choosing an orthonormal basis $e_1,\ldots, e_k$ in $W$,   and extending to an orthonormal basis $e_1,\ldots, e_k, e_{k+1},\ldots, e_{d+1}$ of $V$, we will assume that $V=\mbr^{d+1}$ and $W=\mbr^k\oplus \{0\}$.    Thus the formula we have to establish is
 \begin{equation}\label{E:spheredisint2}
\int_{S^d(a)}f\,d\sigma =\int_{x\in B_k(a)}\left[\int_{y\in S^{d-k}(a_{x})} f(x,y)\,d\sigma(y)\right]\frac{a}{a_{x}}\,dx,
\end{equation}
where $S^d(a)$ is the sphere of radius $a$, centered at $0$, in $\mbr^{d+1}$, and $B_k(a)$ is the ball of radius $a$, center $0$, in $\mbr^k$.

  Let $F$ be the function on $\mbr^{d+1}$ given by
\begin{equation}\label{E:defFf}
F(z)= f\left(\frac{a}{\norm{z}}z\right),
\end{equation}
 with $F(0)$ defined arbitrarily.  Thus $F$ is constant along radial rays and equal to $f$ on the sphere $S^d(a)$.
 
 Let $\psi$ be any non-negative Borel function on $[0,\infty)$. We work out the integral
$$\int_{\mbr^{d+1}}F(z)\psi(\norm{z}^2)\,dz$$
in two ways.   
 
Using the polar disintegration formula (\ref{E:polard}) and scaling (\ref{E:spherscal}) we have
\begin{equation}\label{E:Fpsi}
\begin{split}
 \int_{\mbr^{d+1}}F(z)\psi(\norm{z}^2)\,dz  &=\int_0^\infty\left[\int_{S^d }F(rw) \,r^dd\sigma(w)\right]\,\psi(r^2)\,dr\\
 &=\left(\int_{S^d(a)}f(w)\,d\sigma(w)\right) \int_0^\infty \psi(r^2)(r/a)^d\,dr.
 \end{split}\end{equation}
This expresses the spherical integral on the right in terms of the volume integral on the left.

Next we  will split $\mbr^{d+1}$ into $\mbr^k$ and $\mbr^{d+1-k}$ and disintegrate the left side in (\ref{E:Fpsi}) by repeated use of Fubini's theorem:

\begin{equation}\label{E:intFpsi1}
\begin{split}
&\int_{\mbr^{d+1}}F(z)\psi(\norm{z}^2)\,dz \\
&= \int_{\mbr^k}\left[\int_{\mbr^{d+1-k}}F(x,y)\psi(\norm{x}^2+\norm{y}^2)\,dy\right]\,dx\\
&  = \int_{\mbr^k}\left[\int_{R\in (0,\infty)}\left\{\int_{w\in S^{d-k} } F(x,Rw)\psi(\norm{x}^2+R^2)\,d\sigma(w)\right\}\,R^{d-k}dR \right]\,dx \\
&=\int_{\mbr^k\times (0,\infty)\times S^{d-k}} F(x,Rw)\psi(\norm{x}^2+R^2)\,d\sigma(w) \,R^{d-k}dR \,dx.
\end{split}
\end{equation}
Here we have used the assumption that $W$ is a {\em proper} subspace of $V$, which in the present notation means that $k<d+1$.
Now, for {\em fixed} $x\in\mbr^{k}$, we change variables from $R$ to $r\geq \norm{x}$ given by
\begin{equation}\label{E:r2Rx}
r^2= R^2+\norm{x}^2. 
\end{equation}
Then
\begin{equation}\label{E:rdr}
rdr=RdR.
\end{equation}
Hence, using (\ref{E:intFpsi1}), we have
\begin{equation}
\begin{split}
&\int_{\mbr^{d+1}}F(z)\psi(\norm{z}^2)\,dz \\
&= \int_{x\in \mbr^k, R\in (0,\infty), w\in S^{d-k}, r\geq \norm{x}} F(x,Rw)\psi(\norm{x}^2+R^2)\,{d\sigma(w)} \,R^{d-k}dR \,dx \\
&=\int _{r\in (0,\infty), x\in B_k(r), w\in S^{d-k}} F(x, r_xw) \psi(r^2) \, d\sigma(w)\, r_x^{d-k-1}\,rdr \,dx
\end{split}
\end{equation}
where we have now written $r_x$ for $R$:
\begin{equation}\label{E:defrx}
r_x=\sqrt{r^2-\norm{x}^2}.
\end{equation}
Recalling  the choice of the function $F$, we have:
$$F(z)=f\left(\frac{a}{\norm{z}}z\right)= f\left((a/r)x, (aR/r)w \right)\quad\hbox{ if $z=(x,Rw)$ with $w\in S^{d-k}$.}$$
 
 Thus 
  \begin{equation}\label{E:Fpsidz2}
\begin{split}
&\int_{\mbr^{d+1}}F(z)\psi(\norm{z}^2)\,dz \\
&=\int _{r\in (0,\infty), x\in B_k(r)}\left[\int_{ w\in S^{d-k}} f\left(ax/r, ar_xw/r\right) \, d\sigma(w)\right]\, r_x^{d-k-1}\,dx\, \psi(r^2)rdr. \end{split}
\end{equation}

Keeping in mind that $f$ is  evaluated only at points on the sphere $S^d(a)$, we change coordinates to make clearer use of this.
For fixed   $r$ and $x$, we change from variable $w$ to 
\begin{equation}\label{E:wprimexprime}
w'=\frac{ar_x}{r}w =a_{x'}w,\qquad\hbox{where}\quad  {x'}=\frac{a}{r}x,
\end{equation}
 which changes the spherical integral on the right side of (\ref{E:Fpsidz2})  to
$$a_{x'}^{k-d}\int_{w'\in S^{d-k}(a_{x'})}f(x',  w') \,{d\sigma(w')}.$$
 Thus:
\begin{equation*}
\begin{split}
&\int_{\mbr^{d+1}}F(z)\psi(\norm{z}^2)\,dz \\
&=  \int _{r\in (0,\infty), x\in B_k(r) }\left[\int_{S^{d-k}(a_{x'})} f(x',  w') \,a_{x'}^{k-d}{d\sigma(w')}\right]\, r_x^{d-k-1}\,dx\,\psi(r^2)\,rdr\\
&=  \int _{r\in (0,\infty), x\in B_k(r) }\left[\int_{S^{d-k}(a_{x'})} f(x',  w') 
\,{d\sigma(w')}\right]\, a_{x'}^{k-d} r_x^{d-k-1}\,dx\,\psi(r^2)\,rdr.
 \end{split}
\end{equation*}
 
Note that since $x\in B_k(r)$ we have $x'\in B_k(a)$, and, by (\ref{E:wprimexprime}),
$$dx=(r/a)^{ k}dx'$$
and
$$r_x = \frac{r}{a}a_{x'}.$$
Thus:
\begin{equation}\label{E:intF2}
\begin{split}
&\int_{\mbr^{d+1}}F(z)\psi(\norm{z}^2)\,dz \\
&= \int _{r\in (0,\infty), x'\in B_k(a) }\left[\int_{S^{d-k}(a_{x'})} f(x',  w') \,{d\sigma(w')}\right]\frac{a}{a_{x'}}\ \,dx'\psi(r^2)\,(r/a)^{d }dr.
 \end{split}
\end{equation}
 Choosing  $\psi$ for which 
$$\int_0^\infty \psi(r^2)(r/a)^d\,dr=1,$$
and comparing (\ref{E:intF2}) with the earlier expression  (\ref{E:Fpsi}) we obtain:
\begin{equation}
\begin{split}
 \int_{S^d(a)}f\,d\sigma = \int_{x'\in B_k(a) }\left[\int_{S^{d-k}(a_{x'})} f(x',  w') \,{d\sigma(w')}\right]\frac{a}{a_{x'}} \,dx' .  \end{split}
\end{equation}
 \end{proof}
   
\subsection{A more general spherical disintegration } We leverage the disintegration formula (\ref{T:spherdisint})  to obtain a more general form by allowing projections that are not orthogonal.

 \begin{theorem}\label{T:gendisint}
  Let $V$ be  a finite-dimensional real inner-product space and let 
 $$L:V\to X$$
 be a linear surjection onto a    real inner-product space $X$, where $0<\dim X<\dim V$.  Then $L$ restricts to an isomorphism
 \begin{equation}
 L_0:(\ker L)^{\perp}\to X,
 \end{equation}
and
\begin{equation}\label{E:disintfsigBkLS}
\begin{split}
&\int_{S_V(a)} f\,d\sigma \\
&=\int_{x\in X, \norm{L_0^{-1} x}< a} \left\{\int_{ S_V(a)\cap L^{-1}(x)} f \,d\sigma     \right\}  \frac{a}{\sqrt{a^2-\norm{L_0^{-1}x}^2} }  \frac{dx}{\sqrt{|\det LL^*|}},
\end{split}
\end{equation}
for any  non-negative or bounded Borel function $f$, defined on the sphere $S_V(a)$ of radius $a$ and center $0$ in $V$.
\end{theorem}
Let us observe that if
\begin{equation}\label{E:Pperp}
P_{\perp}:V\to V
\end{equation}
is the orthogonal projection onto $(\ker L)^{\perp}$ then 
\begin{equation}\label{E:LLP0}
L_0P_{\perp}=L.
\end{equation}
We can check this by noting that for any $z\in \ker L$ both sides are $0$ and for any $z\in (\ker L)^\perp$ the left side equals $L_0z$, which, by definition of $L_0$, is $Lz$.
 
\begin{proof}
We use the standard formula for transformation of integrals
\begin{equation}\label{E:phiJint}
\int_{X'}\phi(x')\,dx' = \int_{X}\phi(Jx)\,|\det J|\,dx 
\end{equation}
where $J:X\to X'$ is an isomorphism of a finite-dimensional inner-product space $X$ onto an inner-product space $X'$. This is valid for any Borel function $\phi$ on $X'$ for which either side of (\ref{E:phiJint}) exists.  We apply this with $J=L_0^{-1}:X\to (\ker L)^\perp$ to obtain
\begin{equation}\label{E:phixxprim}
\int_{(\ker L)^\perp} \phi(x')\,dx'= \int_X \phi(L_0^{-1}x)\frac{dx}{ |\det L_0|}.
\end{equation}
The Jacobian term $|\det L_0|$ is computed as the absolute value of the determinant of any matrix of $L_0$ relative to orthonormal bases in $(\ker L)^\perp$ and $X$; in terms of $L$ it is given by:
\begin{equation}\label{E:detL0}
|\det L_0|=\sqrt{|\det LL^*|}.
\end{equation}
Let us note that if $z\in L^{-1}(x)$ then $Lz=x$ and so, with $z_0=L_0^{-1}x\in (\ker L)^\perp$, we have 
\begin{equation}
L(z-z_0)=0,
\end{equation}
and so
\begin{equation}z\in z_0+\ker L.
\end{equation}
Thus any point in $L^{-1}(x)$ is $L_0^{-1}(x)$ plus a vector orthogonal to $z_0$
and so
\begin{equation}\label{E:zz0norm}
\hbox{the element of smallest norm in $L^{-1}(x)$ is $z_0=L_0^{-1}x$.}
\end{equation}
For $\phi$ we use the function on $(\ker L)^\perp$ given by
\begin{equation}\label{E:phixp}
\phi(x')=\frac{a}{a_{x'}} \int_{S_V(a)\cap P_{\perp}^{-1}(x')}f\,d\sigma,
\end{equation}
where 
$ P_{\perp}:V\to V$ 
is the orthogonal projection onto the subspace $(\ker L)^{\perp}$, as in (\ref{E:Pperp}), and the right side in (\ref{E:phixp}) is taken to be $0$ when $\norm{x'}\geq a$. If $f$ is continuous then $\phi$ is continuous on the open ball of radius $a$, and $0$ outside this ball. Then by standard limiting arguments $\phi$ is Borel when $f$ is the indicator function of a compact set,  and hence $\phi$ is Borel for any  non-negative or bounded Borel function $f$.

Then
\begin{equation}\label{E:intL0inv}
\phi(L_0^{-1}x)=\frac{a}{a_{L_0^{-1}x}}\int_{S_V(a)\cap (L_0P_{\perp})^{-1}(x)}f\,d\sigma=\frac{a}{a_{L_0^{-1}x}}\int_{S_V(a)\cap L^{-1}(x)}f\,d\sigma,
\end{equation}
on using the relation (\ref{E:LLP0}). Here, and often, we take the integral over the empty set to be $0$; thus:
$$\hbox{$\phi(L_0^{-1}x)$ is $0$ if $L^{-1}(x)\cap S_V(a)$ is empty.}$$
By (\ref{E:zz0norm}), this means
\begin{equation}\label{E:phizero}
\phi(L_0^{-1}x)=0\qquad\hbox{if $\norm{L_0^{-1}x}>a$.}
\end{equation}
We assume for now that $f\geq 0$; then $\phi\geq 0$. Applying (\ref{E:phixxprim}) and (\ref{E:intL0inv}), we have
\begin{equation}\label{E:phixxprim2}
\begin{split}
&\int_{(\ker L)^\perp}\left[\int_{S_V(a)\cap P^{-1}(x')}f\,d\sigma\right]\frac{a}{a_{x'}}\,dx'\\
&= \int_X \left[\int_{S_V(a)\cap L^{-1}(x)}f\,d\sigma\right]\frac{a}{a_{L_0^{-1}x}}\frac{dx}{ |\det L_0|}.
\end{split}
\end{equation}
The integrand on the left  is $0$ outside the ball of radius $a$ in $(\ker L)^\perp$ and, by (\ref{E:phizero}), the integrand on the right is $0$ unless $\norm{L_0^{-1}x}<a$. By Theorem \ref{T:spherdisint}   the left side is equal to $\int_{S_V(a)}f\,d\sigma$. This proves the identity (\ref{E:disintfsigBkLS}) for $f\geq 0$.

For general complex-valued bounded $f$ the result follows by considering real and imaginary parts and then positive and negative parts. Since $f$ is bounded, all the integrals over $S_V(a)$ involved are finite.
\end{proof}

We are mainly interested in the case where $\dim V$ is large compared to $\dim X$, and, in particular, $m=\dim V-\dim X$ is $\geq 2$. Then in the definition (\ref{E:phixp}) of  $\phi(x')$  the integral of $f$ is over a sphere, of dimension $m-1\geq 1$,  of radius $a_{x'}$, and so, for bounded $f$, the integral is bounded by a constant times $a_{x'}^{m-1}$. Therefore, $\phi(x')$ itself is bounded by a constant times a non-negative power of $a_{x'}$. Thus, $\phi$ is  bounded if $f$ is bounded.

\subsection{A more general slice}   In  the following result $Z$, $W$, and $X$ are  finite-dimensional inner-product spaces, and $\cl:Z\to X$ and $Q:Z\to W$ are  linear surjections.  Figure \ref{F:fullfig} describes the setting of the result.

 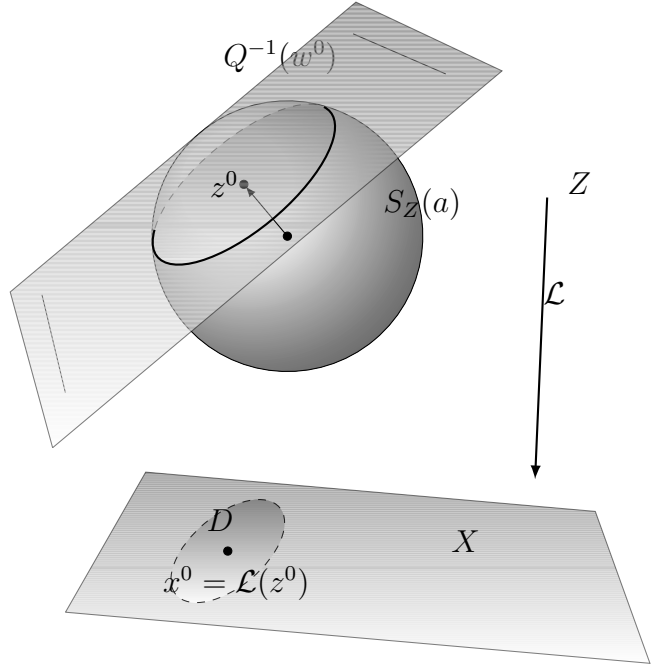
\begin{figure}
\tikzset{
    partial ellipse/.style args={#1:#2:#3}{
        insert path={+ (#1:#3) arc (#1:#2:#3)}
 }}

 {
\centering

   \begin{tikzpicture}   [scale=.6,   rotate= 40]

 
\def\R{3} 
\def\angEl{40} 

\filldraw[ball color=white] (0,0) circle (\R);


 \coordinate [label=left:\textcolor{black}{$Q^{-1}(w^0)$}] (L) at (3.5, 2.25);

 \coordinate [label=left:\textcolor{black}{$S_Z(a)$}] (X) at 
 (3.5, -2);
 
  \coordinate [label=left:\textcolor{black}{$z^0$}] (z0) at 
 (0,1.5);
 

 \coordinate [label=left:\textcolor{black}{${}$}](G) at (-6,.5);

\coordinate [label=left:\textcolor{black}{${}$}] (C) at (5,.5);

\coordinate (F) at (4,2.5);

\coordinate (H) at (-5 ,2.5);

 \draw (G)-- (H);
 
 \draw (C)--(F);


\shadedraw [color=black, opacity = .5]  (-7,-.25)--(6,-.25)--(4.5,3)--(-5.5,3)--(-7,-.25);


 \draw[->, shorten <=1pt,shorten >=1pt](0,0)--(0, 1.5);
 
 
\draw[black,fill=black] (0,0) circle (.5ex);

\draw[black,fill=black] (0,1.5) circle (.5ex); 

 
 \shadedraw [color=black, dashed, opacity=0.3] (0,1.5) ellipse (2.5  and 1);
 
 \draw[thick] (0,1.5) [partial ellipse=150:370: 2.5 and 1];


\shadedraw [color=black, opacity = .5, cm={cos(45) ,-sin(45) ,sin(45) ,cos(45) ,(-4,-8)}]  (-7,-.25)--(6,-.25)--(4.5,3)--(-5.5,3)--(-7,-.25);

 \coordinate [label=left:\textcolor{black}{$X$}] (X) at 
 (-1,-8);


\draw[->, thick, shorten <=1pt,shorten >=1pt](5,-3)--(.7,-7.7);

 \coordinate [label=left:\textcolor{black}{${\mathcal L}$}] (cL) at 
 (4,-5);
 
  
   \shadedraw [color=black, dashed, opacity=0.8] (-5.5,-4.5) ellipse (1.5  and .8);

  \coordinate [label=right: \textcolor{black}{$D$}] (D) at 
 (-5.5,-3.6 );
 
 \coordinate [label=below:\textcolor{black}{$x^0={\mathcal L}(z^0)$}] (cL) at (-5.5,-4.8);
 
 \draw[black,fill=black] (-5.5,-4.5) circle (.5ex); 
 

      \coordinate [label=left:\textcolor{black}{$z^0$}] (z0) at 
 (0,1.5);

 
 \coordinate [label=left:\textcolor{black}{$Z$}] (cL) at 
 (6, -3.5);

\end{tikzpicture}

}
\caption{The affine subspace $Q^{-1}(w^0)\subset Z$ slicing the sphere $S_Z(a)$, the `projection' ${\mathcal L}:Z\to X$,  the points $z^0$, closest in $Q^{-1}(w^0)$ to the center,  and $x^0={\mathcal L}(z^0)$, and the  ellipsoid $D$ which is the projection on $X$ of the slice of the ball by $Q^{-1}(w^0)$.}
\label{F:fullfig}
\end{figure}

We consider the sphere $S_Z(a)$, centered at $0$ and of radius $a>0$, in $Z$. The sphere is sliced along a `circle' by an affine subspace $Q^{-1}(w^0)$, where $w^0$ is some point in $W$.  We denote by $z^0$ the point on  $Q^{-1}(w^0)$  closest to the origin, and
$$x^0={\cl}(z^0)\in X$$
the `projection' of $z^0$ on $X$ by $\cl$.
We will also need the restriction of $\cl$ to the subspace $\ker Q$, and the determinant of this restriction. In more detail, let $L_0$ be the restriction of $\cl$ to  the subspace of $\ker Q$ that is the orthogonal complement of $\ker ({\cl}|\ker Q)$:
$$L_0:\ker Q\ominus\ker \cl\to \cl(\ker Q):z\mapsto {\cl}z,$$
where on the left we have the orthogonal complement of $\ker(\cl|\ker Q)$ within $\ker Q$. As before in (\ref{E:detL0}), the determinant $|\det L_0|$ is the absolute value of the determinant of the matrix of $L_0$ relative to orthonormal bases in its domain and range; we take $|\det L_0|$ to be $1$ in the degenerate case where $L_0$ is $0$.

\begin{theorem}\label{T:slicedisint}    (Figure \ref{F:fullfig})
Let   $f$ a bounded, or non-negative, Borel function defined on  the `circular slice' $S_Z(a)\cap Q^{-1}(w^0)$ for some $w^0\in W$.  Let ${z^0}$ be the point on  $Q^{-1}(w^0)$ closest to $0$, ${x^0}={\cl}{z^0}\in X$. Let $L_0$ be the restriction of $\cl$ to the subspace of $\ker Q$ that is the orthogonal complement of $\ker ({\cl}|\ker Q)$. Then
\begin{equation}\label{E:disintgenslice}
\begin{split}
&\int_{S_Z(a)\cap Q^{-1}(w^0)}f\,d\sigma\\
&=\int_{x\in D}\left\{\int_{S_Z(a)\cap Q^{-1}(w^0)\cap {\cl}^{-1}(x)}f\,d\sigma\right\}\frac{a_{{z^0}} } {\sqrt{a_{{z^0}}^2 -  \norm{L_0^{-1}(x-{x^0})}^2  }}\,\frac{dx}{|\det L_0 |},
\end{split}
\end{equation}
where    $D$   consists of all $x\in x^0+{\cl}(\ker Q)\subset X$ for which the term under the square-root is positive:
\begin{equation}\label{E:defDLx0}
D=x^0+\{y\in {\cl}(\ker Q) :\,    \norm{L_0^{-1}(y)} < a_{{z^0}}\}.
\end{equation}
\end{theorem}

On the left in (\ref{E:disintgenslice}) is the integral of $f$ over the  `circular' slice of the sphere $S_Z(a)$ by the affine subspace $Q^{-1}(w^0)$. On the right is the disintegration of this with respect to the values of $\cl$. In this disintegration each fiber $S_Z(a)\cap Q^{-1}(w^0)\cap {\cl}^{-1}(x)$ is a sphere of radius $\sqrt{a_{{z^0}}^2 -  \norm{L_0^{-1}(x-{x^0})}^2  }$, as we will show following the proof below. The set $D$ is an ``ellipsoid.'' (In the degenerate case where $\cl$ is actually zero on $\ker Q$ the integral over $dx$ drops out and we have a trivial equality in (\ref{E:disintgenslice}).)

Figure \ref{F:fullfig2} illustrates some of the objects involved here.  In the picture, $\ker Q$ is a two-dimensional subspace (through the origin, parallel to $Q^{-1}(w^0)$). Since ${\mathcal L}$ maps $\ker Q$ onto the two-dimensional space $X$, its kernel is, in this picture, just $\{0\}$.

\begin{proof} We will apply the disintegration result Theorem \ref{T:gendisint}, taking for $V$  the subspace $\ker Q\subset Z$, and $L$ the restriction of ${\cl}$ to $V$:
\begin{equation}\label{E: LcL}
L={\cl}|\ker Q: V\to {\cl}(\ker Q).
\end{equation}
Then $L_0$ is, as in  Theorem \ref{T:gendisint}, the restriction:
\begin{equation}
L_0: (\ker L)^\perp  \to {\rm Im}(L)={\cl}(\ker Q),
\end{equation}
where $(\ker L)^\perp$ is the subspace of $V$ consisting of all vectors in $V$ orthogonal to $\ker L$. In more detail,
\begin{equation}\label{E:defKerLp}
(\ker L)^\perp = \{z\in \ker Q: z\in (\ker {\cl}\cap \ker Q)^\perp\}.
\end{equation}

 The center of the `circle' $S_Z(a)\cap Q^{-1}(w^0)$ is the point on $Q^{-1}(w^0)$ closest to $0$. Let us check that this point is given by
\begin{equation}\label{E:p0}
{z^0}=Q^*(QQ^*)^{-1}(w^0);
\end{equation}
here we note that since $Q$ is surjective, $QQ^*$ is invertible because any vector in its kernel would also be in $\ker Q^*=\bigl[{\rm Im}(Q)\bigr]^\perp$.
Clearly,
$$Q{z^0}=w^{0},$$
and if $v\in \ker Q$ then
\begin{equation}\label{E:kerQvz}
\la v, {z^0}\ra =\la Qv, (QQ^*)^{-1}(w^0)\ra=0.
\end{equation}
This implies that the point ${z^0}+v\in Q^{-1}(w^0)$ has norm
$$\norm{v+{z^0}}^2=\norm{v}^2+\norm{{z^0}}^2\geq \norm{{z^0}}^2,$$
thus showing that ${z^0}$ is the unique point on $Q^{-1}(w^0)$ closest to the origin.

 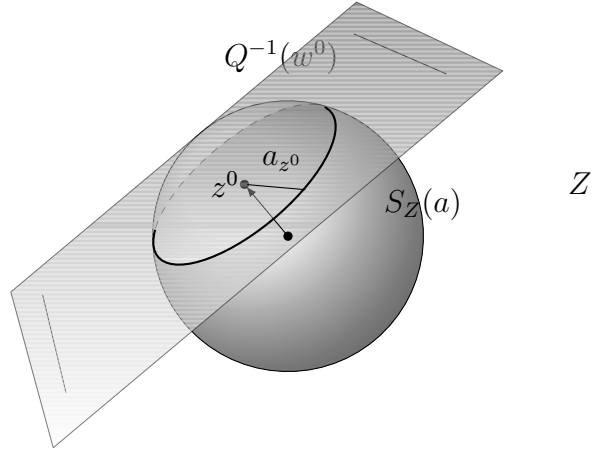
\begin{figure}
\tikzset{
    partial ellipse/.style args={#1:#2:#3}{
        insert path={+ (#1:#3) arc (#1:#2:#3)}
 }}

 {
\centering

   \begin{tikzpicture}   [scale=.6,   rotate= 40]

 
\def\R{3} 
\def\angEl{40} 

\filldraw[ball color=white] (0,0) circle (\R);


 \coordinate [label=left:\textcolor{black}{$Q^{-1}(w^0)$}] (L) at (3.5, 2.25);

 \coordinate [label=left:\textcolor{black}{$S_Z(a)$}] (X) at 
 (3.5, -2);
 
  \coordinate [label=left:\textcolor{black}{$z^0$}] (z0) at 
 (0,1.5);
 

 \coordinate [label=left:\textcolor{black}{${}$}](G) at (-6,.5);

\coordinate [label=left:\textcolor{black}{${}$}] (C) at (5,.5);

\coordinate (F) at (4,2.5);

\coordinate (H) at (-5 ,2.5);

 \draw (G)-- (H);
 
 \draw (C)--(F);


\shadedraw [color=black, opacity = .5]  (-7,-.25)--(6,-.25)--(4.5,3)--(-5.5,3)--(-7,-.25);


 \draw[->, shorten <=1pt,shorten >=1pt](0,0)--(0, 1.5);
 
 
\draw[black,fill=black] (0,0) circle (.5ex);

\draw[black,fill=black] (0,1.5) circle (.5ex); 

 
 \shadedraw [color=black, dashed, opacity=0.3] (0,1.5) ellipse (2.5  and 1);
 
 \draw[thick] (0,1.5) [partial ellipse=150:370: 2.5 and 1];

 
 \draw[shorten <=1pt,shorten >=1pt](0,1.5)--(1, 0.5);


 \coordinate [label=left:\textcolor{black}{$a_{z^0}$}] (az0) at 
 (1.4, 1 );
 
      \coordinate [label=left:\textcolor{black}{$z^0$}] (z0) at 
 (0,1.5);

 
 \coordinate [label=left:\textcolor{black}{$Z$}] (cL) at 
 (6, -3.5);

\end{tikzpicture}

}
\caption{The affine subspace $Q^{-1}(w^0)\subset Z$ slices the sphere $S_Z(a)$ in a `circle' with center $z^0$ and radius $a_{z^0}$.}
\label{F:fullfig2}
\end{figure}

To disintegrate
$$\int_{S_Z(a)\cap Q^{-1}(w^0)}f\,d\sigma$$
we  write this as an integral over the sphere of radius $a_{z^0}$ in $V=\ker Q$:
\begin{equation}\label{E:fSZQi}
\int_{S_Z(a)\cap Q^{-1}(w^0)}f\,d\sigma=\int_{S_V(a_{{z^0}})}f(z+{z^0})\,d\sigma(z),
\end{equation}
which we see by observing that
\begin{equation}\label{E:SZV}
\begin{split}
S_V(a_{{z^0}})+{z^0}&=\{v+{z^0}\,:\, v\in \ker Q, \norm{v}^2 = a^2-\norm{{z^0}}^2\}\\
&= \{v+{z^0}\,:\, Q(v+{z^0})=w^0 , \norm{v+{z^0}}^2= a^2 \}\\
&= S_Z(a)\cap Q^{-1}(w_0),
\end{split}
\end{equation}
where in the second equality we used the orthogonality (\ref{E:kerQvz}). 
 
Applying the disintegration formula (\ref{E:disintfsigBkLS}) for $\cl$ in (\ref{E:fSZQi}) we obtain:
\begin{equation}\label{E:disintgenslice2}
\begin{split}
& \int_{S_Z(a)\cap Q^{-1}(w^0)}f\,d\sigma \\
&=\int_{y\in D_0}  \left\{\int_{S_V(a_{{z^0}})\cap {L}^{-1}(y)}f(z+z_0)\,d\sigma(z)\right\}\frac{a_{{z^0}} } {\sqrt{a_{{z^0}}^2 -  \norm{L_0^{-1}y}^2  }}\,\frac{dy}{|\det L_0 |},
\end{split}
\end{equation}
where $D_0$ is the set of all $y\in {\rm Im}(L_0)={\cl}(\ker Q)$ for which  $\norm{L_0^{-1}y}<a$. Changing variables by translation with $y=x-x^0$, we then have
\begin{equation}\label{E:disintgenslice3}
\begin{split}
 \int_{S_Z(a)\cap Q^{-1}(w^0)}f\,d\sigma = 
& \int_{x\in D} I(x)     \frac{a_{{z^0}} } {\sqrt{a_{{z^0}}^2 -  \norm{L_0^{-1}(x-x^0)}^2  }}\,\frac{dx}{|\det L_0 |},
\end{split}
\end{equation}
where $D=D_0+x^0$ and
\begin{equation}\label{E:disintgenslice4}
\begin{split}
I(x) &=  \int_{S_V(a_{{z^0}})\cap {L}^{-1}(x-x^0)}f(z+z_0)\,d\sigma(z)\\
&= \int_{[S_V(a_{{z^0}})\cap {L}^{-1}(x-x^0)]  +z^0}f(z)\,d\sigma(z).
\end{split}
\end{equation}
Now
 \begin{equation}\label{E:interscirc}
 \begin{split}
 [S_V(a_{{z^0}})\cap {L}^{-1}(x-x^0)]+{z^0} &=[S_V(a_{{z^0}})+z^0]\cap {\cl}^{-1}(x),
 \end{split}
 \end{equation}
because a point $p$ lies in the right hand side if and only if $p=p^0+z^0$, where $p^0\in S_V(a_{{z^0}})$ and ${\cl}(p^0)=x-{\cl}(z^0)=x-x^0$. Then, using (\ref{E:SZV}), we have
\begin{equation}\label{E:interscirc2}
 \begin{split}
 [S_V(a_{{z^0}})\cap {L}^{-1}(x-x^0)]+{z^0} & = S_Z(a)\cap Q^{-1}(w_0)\cap {\cl}^{-1}(x).
 \end{split}
 \end{equation}
Hence
\begin{equation}
I(x) = \int_{ S_Z(a)\cap Q^{-1}(w_0)\cap {\cl}^{-1}(x)}  f(z)\,d\sigma(z).
\end{equation}
Using this value of $I$ in (\ref{E:disintgenslice3}) gives us the desired disintegration formula  (\ref{E:disintgenslice}).  \end{proof} 

The left side of (\ref{E:interscirc}) is the translate by $z^0$ of the intersection of the sphere $S_V(a_{z^0})$ with the affine subspace $L^{-1}(x-x^0)$. The point on this affine subspace closest to the origin is  $L_0^{-1}(x-x^0)$ (by our note in (\ref{E:zz0norm})). Therefore, the radius of the ``circle'' of intersection is given by:
\begin{equation}\label{E:rad}
{\rm radius\,of\,} \left[S_V(a_{{z^0}})\cap {L}^{-1}(x-x^0)\right]=\sqrt{a_{z^0}^2- \norm{ L_0^{-1}(x-x^0)}^2}.
\end{equation}
Hence the same is true of the translate of this circle by $z^0$:
\begin{equation}\label{E:rad2}
{\rm radius\,of\,} \left[S_Z(a)\cap Q^{-1}(w_0)\cap {\cl}^{-1}(x)\right]=\sqrt{a_{z^0}^2- \norm{ L_0^{-1}(x-x^0)}^2}.
\end{equation}

\subsection{Disintegration of slices expressed in coordinates}
Now let us work out some details of the disintegration of slices formula (\ref{E:disintgenslice}).   We apply Theorem \ref{T:slicedisint}   with $Z=\mbr^{d+1}$, $X=\mbr^k$, where $0<k<d$,  and with $\cl$ being the projection ${\cl}=P_{(k)}:\mbr^{d+1}\to\mbr^k$.

  Suppose $u_1,\ldots, u_m$ form an orthonormal basis of $(\ker Q)^\perp$. Then
 $$\hbox{$Qu_1,\ldots, Qu_m$ form a basis of $W={\rm Im}(Q)$.}
$$
Thus
\begin{equation}\label{E:Qz}
Qz=\la z, u_1\ra Qu_1+\ldots +\la z, u_m\ra Qu_m \quad\hbox{for all $z\in (\ker Q)^\perp$.}
\end{equation}
As before, let
$$V=\ker Q,$$
and $L$ the restriction of the projection ${\cl}$ to $V$:
\begin{equation}\label{E:defLdk}
L:V\to X=\mbr^k: z\mapsto z_{(k)} \stackrel{\rm def}{=} (z_1,\ldots, z_k).
\end{equation}
Let us note that  
$$\ker {\cl}=\{0\}\oplus {\mbr}^{d+1-k},$$
and, of more interest,
\begin{equation}\label{E:kerL}
\begin{split}
\ker L&=\{(0,y)\in\mbr^{d+1}\,:\, \la y, (u_1)_{(k)'}\ra =0,  \ldots, \la y, (u_m)_{(k)'}\ra=0 \}\\
&=\{0\}\oplus \left[\hbox{span of $(u_1)_{(k)'},\ldots, (u_m)_{(k)'}$}\right]^\perp
\end{split}
\end{equation}
where $(u_i)_{(k)'} = ((u_i)_{k+1}, \ldots, (u_i)_{d+1}) \in \mathbb{R}^{d+1-k}$. 
The space $(\ker L)^\perp$ consists of all $z\in \ker Q$ that are orthogonal to the subspace $\ker L$. Thus a vector $z=(x,y)\in \mbr^{d+1}$ lies in $(\ker L)^\perp$ if and only if   $z\in \ker Q$ and the component $x$ is unrestricted but the component $y$ is  orthogonal to $ \left[\hbox{span of $(u_1)_{(k)'},\ldots, (u_m)_{(k)'}$}\right]^\perp$:
\begin{equation}\label{E:kerLperp}
(\ker L)^{\perp}=\{z=(x,y)\in \mbr^k\oplus  \left[\hbox{span of $(u_1)_{(k)'},\ldots, (u_m)_{(k)'}$}\right]\,:\, z\in \ker Q\}.
\end{equation}
Thus    $(\ker L)^{\perp}$  consists of all elements of $\mbr^{d+1}$ of the form
$$(x,0)+(0,c_{ 1}(u_1)_{(k)'}+\ldots +c_{m}(u_m)_{(k)'})$$
that are orthogonal to $u_1,\ldots, u_m$:
\begin{equation}
\la (u_a)_{(k)}, x\ra +\sum_{b= 1}^m  \la (u_a)_{(k)'}, (u_b)_{(k)'}\ra  c_b=0
\end{equation}
for $a\in\{1,\ldots, m\}$. These $m$ equations yield a solution for $(c_1,\ldots, c_m)$:
\begin{equation}\label{E:vecc}
c=U^{-1}{\vec x}
\end{equation}
where $c=(c_1,\ldots, c_m)$, the linear mapping
$$U:\mbr ^m\to \mbr^m$$
has matrix
$$[\la (u_a)_{(k)'}, (u_b)_{(k)'}\ra],$$
and
$${\vec x}=(\la (u_1)_{(k)}, x\ra,\ldots, \la  (u_m)_{(k)}, x\ra)\in\mbr^m.$$
The mapping $L$ restricted to $(\ker L)^\perp$ is given by
\begin{equation}\label{E:LXkerP}
\begin{split}
L_0:(\ker L)^{\perp}  &\to X={\mbr}^k\\
  (x, c_{ 1}(u_1)_{(k)'}+\ldots +c_{m}(u_m)_{(k)'}) & \mapsto x.
\end{split}
\end{equation}
The inverse of this mapping is given by
\begin{equation}\label{E:LXinv}
L_0^{-1}x=\left (x, c_{ 1}(u_1)_{(k)'}+\ldots +c_{m}(u_m)_{(k)'}\right),
\end{equation}
where $(c_1,\ldots, c_m)$ is given by (\ref{E:vecc}).

Next we work out the adjoint ${L_0}^*$. For any $z\in (\ker L)^\perp$, which is the subspace of $\ker Q$ orthogonal to $\ker Q\cap \ker {\cl}$,  we have 
\begin{equation}
\label{E:LXstarm}
\begin{split}
\la L_0^*x, z\ra &=\la x, {L_0}z\ra\\
&= \la x, z_{(k)} \ra\\
&=\la (x,0), z\ra\\
&=\la (x,0) ,  P_{\ker Q}z\ra\\
&=\la P_{\ker Q}(x,0), z\ra\\
&=\la (I-P_{(\ker Q)^\perp})(x,0), z\ra \\
&=\left\langle (x,0) -\sum_{a=1}^m \la x, (u_a)_{(k)}\ra   u_a, z\right\rangle.
\end{split}
\end{equation}
The element 
\begin{equation}\label{E:1minQ}
P_{\ker Q}(x,0) =(x,0) -\sum_{a=1}^m \la x, (u_a)_{(k)}\ra   u_a
 \end{equation}
 lies in  $\ker Q$ and is also in the subspace 
\begin{equation}\label{E:kerQ}
\mbr^k\oplus  \left[\hbox{span of $(u_1)_{(k)'},\ldots, (u_m)_{(k)'}$}\right].
\end{equation}
Thus it is in $(\ker L)^\perp$. Hence
\begin{equation}\label{E:LXstarx}
L_0^*x= P_{\ker Q}(x,0) =(x,0) -\sum_{a=1}^m \la x, (u_a)_{(k)}\ra   u_a.
\end{equation}
From this and the fact that $L_0$ is just the projection onto the first $k$ coordinates, we have
\begin{equation}\label{E:LXstarLX}
{L_0L_0^*}x= x-\sum_{a=1}^m \la x, (u_a)_{(k)}\ra   (u_a)_{(k)}.
\end{equation}
For future reference let us rewrite this in different notation:
\begin{equation}\label{E:L0L0star1}
\begin{split}
L_0L_0^*x= x-\sum_{a=1}^m\la P_{(k)}^*x, u_a\ra P_{(k)}u_a &=\left(I-\sum_{a=1}^mP_{(k)}P_{u_a}P_{(k)}^* \right)x\\
&=\left(I-P_{(k)}P_{(\ker Q)^\perp}P_{(k)}^*\right)x\\
&=P_{(k)}P_{\ker Q}P_{(k)}^*x
\end{split}
\end{equation}
where 
$$P_{(k)}:\mbr^{d+1}\to \mbr^k: z\mapsto z_{(k)}$$
is the projection onto the first $k$ coordinates.

Now let
\begin{equation}\label{E:Pka}
P_{k,a}:\mbr^k\to\mbr^k : x\mapsto \la x, \widehat{(u_a)_{(k)}}\ra \widehat{(u_a)_{(k)}}
\end{equation}
be the orthogonal projection onto the ray spanned by $(u_a)_{(k)}$, assumed to be nonzero. Then
\begin{equation}\label{E:LXstarLX2}
{L_0L_0^*}=I-\sum_{a=1}^m \norm{(u_a)_{(k)}}^2P_{k,a}.
\end{equation}

Now recall the disintegration formula (\ref{E:disintgenslice}):
\begin{equation}\label{E:disintgensliceb}
\begin{split}
&\int_{S^d(a)\cap Q^{-1}(w^0)}f\,d\sigma\\
&=\int_{x\in D}\left\{\int_{S^d(a)\cap Q^{-1}(w^0)\cap {\cl}^{-1}(x)}f\,d\sigma\right\}\frac{a_{{z^0}}}{\sqrt{a_{{{z^0}}}^2 -  \norm{L_0^{-1}(x-{x^0})}^2  }}\,\frac{dx}{|\det L_0 |}
\end{split}
\end{equation}
where 
\begin{equation}\label{E:x0}
{x^0}={\cl}{z^0},
\end{equation}
and
\begin{equation}\label{E:defDL2}
D=x^0+\{y\in {\cl}\bigl(\ker Q\bigr)\,:\,    \norm{L_0^{-1}(y)}\leq a_{{{z^0}}}\}.
\end{equation}
We have now both a way to compute ${L_0}^{-1}$, given in (\ref{E:LXinv}), and an expression for the determinant factor:
\begin{equation}\label{E:detLXfactpr}
|\det L_0| =\sqrt{\det\left(I-\sum_{a=1}^m \norm{(u_a)_{(k)}}^2P_{k,a}\right)}.
\end{equation}

\subsection{Integrals of functions on subspaces}
We consider now a function $f$  on $\mbr^{d+1}=\mbr^k\oplus \mbr^{d+1-k}$ that depends only on the first $k$ components:
$$f(x,y)=\phi(x).$$
We denote by $P_{(k)}$ the projection onto the first $k$ coordinates:
\begin{equation}
P_{(k)}z=z_{(k)}=(z_1,\ldots, z_k)\in\mbr^k.
\end{equation}
For   convenience let us assume that $P_{(k)}$ maps $\ker Q$ onto $\mbr^k$. 
Then, applying the disintegration formula of Theorem \ref{T:slicedisint}, we have
\begin{equation}\label{E:disintgenslice3ca}
\begin{split}
&\int_{S^d(a)\cap Q^{-1}(w^0)}f\,d\sigma\\
&=\int_{x\in D}
\phi(x) V_a(x)
\frac{a_{{{z^0}}}}{\sqrt{a_{{z^0}}^2 -  \norm{L_0^{-1}(x-{x^0})}^2  }}\,\frac{dx}{|\det L_0 |}
\end{split}
\end{equation}
wherein   $D$ is the set of all $x\in\mbr^k$ for which the term under $\sqrt{\ldots}$ is positive, and
\begin{equation}
V_a(x)={\rm Vol}\left(S^d(a)\cap Q^{-1}(w^0)\cap P_{(k)}^{-1}(x) \right),
\end{equation}
 is the volume of the $(d-m-k)$-dimensional sphere of 
radius  given by (\ref{E:rad2}): 
\begin{equation}\label{E:radslice}
\sqrt{a_{{{z^0}}}^2- \norm{L_0^{-1}(x-{x^0})}^2}.
\end{equation}
Using (\ref{E:LXstarLX2}) we have
\begin{equation}\label{E:LXinvsq}
\norm{L_0^{-1}w^0}^2=\la ({L_0L_0^*})^{-1}w^0,w^0\ra  =\left\langle \left(I-\sum_{a=1}^m \norm{(u_a)_{(k)}}^2P_{k,a}\right)^{-1}w^0, w^0\right\rangle.
\end{equation}
The volume, or  `surface area', in the integrand on the right in (\ref{E:disintgenslice3ca}) is therefore:
\begin{equation}\label{E:volsphere}
\begin{split}
& {\rm Vol}\left(S^d(a)\cap Q^{-1}(w^0)\cap {P_{(k)}}^{-1}(x) \right)\\
&=   c_{d-k-m} \left[a_{{{z^0}}}^2- \norm{L_0^{-1}(x-{x^0})}^2 \right]^{\frac{d-k-m}{2}}
 \end{split}
\end{equation}
where $c_{d-k-m}$ is the surface measure of the $(d-k-m)$-dimensional sphere given, for all $j$, by the formula:
\begin{equation}\label{E:cjsurf}
c_{j}= 2\frac{\pi^{\frac{j+1}{2}}}{\Gamma\left(\frac{j+1}{2}\right)}.
\end{equation}
We can then rewrite (\ref{E:disintgenslice3ca}) as
\begin{equation}\label{E:disintgenslice3b}
\begin{split}
 \int_{S^d(a)\cap Q^{-1}(w^0)}f\,d\sigma
&=c_{d-k-m}\int_{x\in D}
I'(x)\,\frac{dx}{|\det L_0 |},
\end{split}
\end{equation}
where
\begin{equation}\label{E:sliceIN}
I'(x)= \phi(x)a_{z^0}    \left[a_{{{z^0}}}^2- \norm{L_0^{-1}(x-{x^0})}^2 \right]^{\frac{d-k-m-1}{2}}.
\end{equation}
The sphere $S^d(a)\cap Q^{-1}(w^0)$ has dimension $d-m$ and its volume (``surface area'') is
$$c_{d-m}a_{z^0}^{d-m}.$$
So, using the {\em normalized surface measure} $\ovs$ on the sphere $S^d(a)\cap Q^{-1}(w^0)$, we have
\begin{equation}\label{E:disintgenslice3c}
\begin{split}
 \int_{S^d(a)\cap Q^{-1}(w^0)}f\,d\ovs
&=\frac{c_{d-k-m}}{c_{d-m}a_{z^0}^{d-m} } \int_{x\in D}
I'(x)\,\frac{dx}{|\det L_0 |},
\end{split}
\end{equation}
where $I'(x)$ is as in (\ref{E:sliceIN}).

\section{Limit of spherical integrals}

In this section we prove  Theorem \ref{T:RadnonNlim} (expressing the limit of spherical integrals as Gaussian integration on an affine subspace)  by using the spherical disintegration formula (\ref{E:kerLperp}).

Let $\phi$ be a bounded measurable function on $\mbr^k$. Then, for any $d>k$ we have the function
  $f$  on $\mbr^{d+1}=\mbr^k\oplus \mbr^{d+1-k}$ that depends only on the first $k$ components:
$$f(x,y)=\phi(x).$$
In fact, identifying $\mbr^k$ with the subspace  $\mbr^k\oplus\{0\}$ of $l^2$, we have the function $f$ on $l^2$:
$$f(x,0,0,\ldots)=\phi(x).$$
Let
$$Q:l^2\to\mbr^m $$
be a continuous linear surjection, and  $u_1,\ldots, u_m$   an orthonormal basis of $(\ker Q)^\perp$. The point ${z^0}\in \mbr^N$ closest to $Q^{-1}(w^0)$, where $w^0\in \mbr^m$, is
\begin{equation}
z^0=Q^*(QQ^*)^{-1}(w^0),
\end{equation}
as we have seen before in the context of (\ref{E:kerQvz}). 
Let
\begin{equation}
(\mbr^k\oplus\{0\}) \cap \ker Q =\{z\in \ker Q\,:\, z_{k+1}=z_{k+2}=\ldots =0\}.
\end{equation}

 Our goal in this section is the following result.

   \begin{theorem}\label{T:limintfsig} Consider an affine subspace of $l^2$ given by $Q^{-1}(w^0)$, where $Q:l^2\to W$ is a linear surjection onto a finite-dimensional inner-product space $W$. 
 Suppose that the projection  $P_{(k)}: l^2\to \mbr^k: z\mapsto z_{(k)}$  maps $\ker Q$ onto $\mbr^k$. Let $S_{Z_N}(a)$ be the sphere of radius $a$ in the subspace $Z_N=\mbr^N\oplus\{0\}$ in $l^2$.  Let $\phi$ be a bounded Borel function on $\mbr^k$ and let $f$ be the function obtained by extending $\phi$ to $l^2$ by setting
   $$f(x)=\phi(x_1,\ldots, x_k)\qquad\hbox{for all $x\in l^2$.}$$
     Then
   \begin{equation}\label{E:disintgenslice0}
\begin{split}
&\lim_{N\to\infty} \int_{S_{Z_{N}}(\sqrt{N})\cap Q_N^{-1}(w^0)}f\,d{\bar\sigma}\\
&=(2\pi)^{-k/2} \int_{x\in \mbr^k}
\phi(x) \exp\left({- \frac{  \la ({L_{0} }{L_{0}^* })^{-1}(x-{z^{0}}_{(k)}), x-{z^{0}}_{(k)} \ra}{2} }\right)\,  \frac{dx}{  \sqrt{\det ({L_{0} L_{0} ^*})} },
\end{split}
\end{equation}
where $L_0$ is the restriction of the  projection $P_{(k)}$ to $\ker Q\ominus \ker  P_{(k)}$. 
           \end{theorem}
           As before, the notation $\ker Q\ominus \ker  P_{(k)}$ means the orthogonal complement of $\ker Q\cap \ker P_{(k)}$ within $\ker Q$. Thus $L_0$ is  the restriction of $z\mapsto z_{(k)}$ to the subspace of $\ker Q$ orthogonal to  $\ker Q\cap \ker P_{(k)}$.

           The remainder of this section is devoted to the proof of this result.  If $z\mapsto z_{(k)}$ maps $\ker Q$ onto a proper subspace  $X$ of $\mbr^k$ then (\ref{E:disintgenslice0}) holds with the integral being over the image $X$, and with $L_0L_0^*$ taken as a map $X\to X$.
 
\subsection{Approximating by finite-dimensional subspaces}

Let
\begin{equation}\label{E:JN1}
J_N: \mbr^N\oplus \{0\}\to l^2
\end{equation}
be the inclusion map. Then the adjoint
\begin{equation}\label{E:JN1st}
J_N^*: l^2\to \mbr^N\oplus \{0\} 
\end{equation}
is the orthogonal projection onto the subspace $\mbr^N\oplus\{0\}$, and the composition
\begin{equation}\label{E:QNl2}
 J_NJ_N^*:l^2\to l^2
\end{equation}
is the same orthogonal projection, but now viewed as an operator in $l^2$ whose range is the subspace 
$\mbr^N\oplus \{0\} $. Now let
\begin{equation}\label{E:QN1}
Q_N=QJ_N:\mbr^N\oplus\{0\}\to W.
\end{equation}
By Proposition   \ref{P:limHZN2}, $Q_N$ is surjective for large $N$. Then 
\begin{equation}\label{E:z0N}
z^{0,N}\stackrel{\rm def}{=}Q_N^*(Q_NQ_N^*)^{-1}(w^0)
\end{equation}
is the point on the affine subspace 
$$Q_N^{-1}(w^0)\subset \mbr^N\oplus\{0\}$$
that is closest to $0$.  We now show that $P_{(k)}z^{0,N}\to P_{(k)}z^0$, the point on $Q^{-1}(w^0)$ closest to $0$, as $N\to\infty$, where
\begin{equation}\label{E:Pk}
P_{(k)}:l^2\to X= \mbr^k:z \mapsto z_{(k)}=(z_1,\ldots, z_k).\end{equation}

\begin{prop}\label{P:limz0NK} Let $H$ be a   Hilbert space  and   $Q:H\to W$  a continuous linear surjection onto    a finite-dimensional   space $W$.  Suppose $Z_1\subset Z_2\subset\ldots$ are   finite-dimensional subspaces of $H$ such that $ \cup_{N\geq 1}Z_N$ is dense in $H$.    Then for any $w^0\in W$,
\begin{equation}\label{E:z0Nlim}
\lim_{N\to\infty}{\cl}z^{0,N}={\cl}z^0,
\end{equation}
where $z^{0,N}$ is the point on $Z_N\cap Q^{-1}(w^0)$ closest to $0$, and $z^0$ is the point on $Q^{-1}(w^0)$ closest to $0$, and ${\cl}:H\to X$ is a continuous linear mapping to any finite-dimensional space $X$.
\end{prop}
\begin{proof} Let
\begin{equation}\label{E:defJN}
J_N:Z_N\to H
\end{equation}
be the inclusion map. We equip the finite-dimensional space $W$ with an inner-product. The adjoint
\begin{equation}\label{E:defJNstar}
J_N^*:H\to Z_N
\end{equation}
is the orthogonal projection onto the subspace $Z$. The composition
\begin{equation}\label{E:defPN}
P_{Z_N}=J_NJ_N^*:H\to H
\end{equation}
is again the orthogonal projection in $H$ whose range is the subspace $Z_N$.

 Then with
\begin{equation}\label{E:QN}
Q_N=QJ_N:Z_N\to W,
\end{equation}
we have 
\begin{equation}
Q_NQ_N^*:W\to W,
\end{equation}
which, for large $N$, is an isomorphism 
(Proposition \ref{P:limHZN2}). The point $z^{0,N}$  on $Q_N^{-1}(w^0)$ closest to the origin is given by 
\begin{equation}\label{E:z0N2}
z^{0,N}=Q_N^*(Q_NQ_N^*)^{-1}(w^0).
\end{equation}
For any $z\in H$, 
\begin{equation}\label{E:PkJN}
\begin{split}
\la z,  z^{0,N}\ra &=\la z, Q_N^*(Q_NQ_N^*)^{-1}w^0\ra\\
&=\la z, J_N^*Q^*(QJ_NJ_N^*Q^*)^{-1}w^0\ra\\
&=\la QJ_Nz, (QP_{Z_N}Q^*)^{-1}w^0\ra\\
&=\la Qz, (QP_{Z_N}Q^*)^{-1}w^0\ra.
\end{split}
\end{equation}
Since $\cup_{N\geq 1}Z_N$ is dense in $H$, Lemma \ref{L:incsub} shows that $P_{Z_N}\to I$ {\em pointwise}, as $N\to\infty$, and so, using continuity of $Q$ we have:
\begin{equation}
\lim_{N\to\infty}QP_{Z_N}Q^*y =QQ^*y \qquad\hbox{for all $y\in W$.}
\end{equation}
 {\em Since $W$ is finite-dimensional} this implies
 that
 \begin{equation}
\lim_{N\to\infty}QP_{Z_N}Q^*  =QQ^*  \quad\hbox{in the finite-dimensional space $\End(W)$.}
\end{equation}
Then by continuity of inverses in the open subset of $\End(W)$ consisting of invertible automorphisms we have
 \begin{equation}\label{E:invcont}
\lim_{N\to\infty}(QP_{Z_N}Q^*)^{-1}  =(QQ^*)^{-1}  \quad\hbox{in the finite-dimensional space $\End(W)$.}
\end{equation}
Returning to (\ref{E:PkJN}) we conclude:
 \begin{equation}\label{E:limPkJN}
\lim_{N\to\infty} \la z,z^{0,N}\ra = \la Qz, (QQ^*)^{-1}w\ra =\la z, z^0\ra.
\end{equation}
(Thus, $  z^{0,N} \to z^0$ as $N\to\infty$, as a weak limit.) If $\cl: H\to X$ is continuous linear and $X$ is finite-dimensional then for any $f\in X^*$ we have
 \begin{equation}\label{E:limPkJN2}
  \la f,{\cl}z^{0,N}\ra = \la {\cl}^*f, z^{0,N}\ra\to \la {\cl}^*f, z^0\ra=\la  f, {\cl}z^0\ra,
\end{equation}
which implies, since $X$ is finite-dimensional, that $ {\cl}z^{0,N}\to  {\cl}z^0$.
 \end{proof}

\subsection{The disintegration formula reorganized} We will apply the slice disintegration formula  in Theorem \ref{T:slicedisint} to the space
$$Z_N=\mbr^N\oplus \{0\}\subset l^2,$$
 with the linear mapping
$$Q_N=QJ_N:Z_N\to W,$$ 
 and with  the linear surjection ${\cl}:Z_N\to X$ being the projection
$${\cl}_{N,k}={P}_{(k)}J_N:Z_N\to X=\mbr^k: z\mapsto z_{(k)},$$
where $N>k$ and
\begin{equation}
P_{(k)}:l^2\to \mbr^k:z\mapsto z_{(k)}
\end{equation}
being the projection on the first $k$ components. We will assume that $Q_N$ is a surjection, by considering large $N$ (Proposition \ref{P:limHZN2}).   We apply Theorem \ref{T:slicedisint} to the sphere of radius
$$a=\sqrt{N}$$
in $Z_N$.  The dimension $d$ of the sphere $S_{Z_N}(a)$ is
\begin{equation}\label{E:dZN}
d=\dim Z_N-1=N-1.
\end{equation}

The disintegration formula (\ref{E:disintgenslice3c}), using  the {\em normalized volume measure} ${\bar\sigma}$ on the sphere, is:
\begin{equation}\label{E:disintgenslice4c2}
 \int_{S_{Z_N}(a)\cap Q_N^{-1}(w^0)}f\,d{\bar\sigma} =
\frac{  c_{d-k-m}}{a_{{z^{0,N}}}^{d-m}c_{d-m}} \int_{x\in D_N}
I'_N(x)\,\frac{dx}{|\det L_{0,N} |},
\end{equation}
where  \begin{equation}\label{E:defDL}
D_N=\{x\in \mbr^k:\,    \norm{L_{0,N}^{-1}(x- {z^{0,N}}_{(k)} )} < a_{{z^{0,N}}}\},
\end{equation}
and the integrand is
\begin{equation}\label{E:disintI}
I'_N(x)=\phi(x) 
 a_{{z^{0,N}}} \left\{a_{{z^{0,N}}}^2 -  \norm{{L_{0,N} }^{-1}(x- {z^{0,N}}_{(k)}   )}^2  \right\}^{\frac{d-k-m-1}{2}}.
 \end{equation}

Here $L_{0,N}$ is the restriction of the projection
$$P_{N,k}:Z_N\to\mbr^k:z\mapsto z_{(k)}$$
to  the subspace of $\ker Q_N$ that is the orthogonal complement  within $\ker Q_N$ of $\ker (P_{N,k}|\ker Q_N)$. Proposition \ref{P:L0N}, proven below, ensures that $L_{0,N}$ is is an isomorphism for large $N$ (see the note following the statement of Proposition \ref{P:L0N}).

 Our goal is to work out the limit of the right side of (\ref{E:disintgenslice4}) as $N\to\infty$.

\subsection{First steps towards the large-$N$ limit}
Factoring $a_{{z^{0,N}}}$ out of the various terms on the right  in (\ref{E:disintgenslice4}) leads to 
\begin{equation}\label{E:disintgenslice4b}
 \int_{S_{Z_N}(a)\cap Q_N^{-1}(w^0)}f\,d{\bar\sigma} =
\frac{  c_{d-k-m}}{a_{{z^{0,N}}}^{ k}c_{d-m}} \int_{\mbr^k}
I_N(x)\,\frac{dx}{|\det L_{0,N} |},
\end{equation}
where
\begin{equation}\label{E:IN2}
I_N(x)=  \phi(x) 
  \left\{	1 - a_{{z^{0,N}}}^{-2} \norm{{L_{0,N} }^{-1}(x-{z^{0,N}}_{(k)})}^2  \right\}^{\frac{d-k-m-1}{2}}1_{D_N}(x)\end{equation}
  and the set $D_N$ is, as before, comprised of all points $x$ for which the term within $\{\ldots\}$ is  non-negative.
 Before analyzing the integrand let us work out the limit of the constant term outside the integral on the  right side of (\ref{E:disintgenslice4b}).

Let $N_0$ be a value of $N$ for which $Q_N$ is surjective and $P_{N,k}(\ker Q_N)=\mbr^k$. Then for $N\geq N_0$,
\begin{equation}\label{E:az0Nineq}
\left(N-\norm{z^{0,N_0}}^2\right)^{k/2}\leq a_{z^{0,N}}^k \leq \left(N-\norm{z^0}\right)^{k/2},
\end{equation}
because the point $z^{0,N}$ is at most as far from the origin as $z^{0,N_0}$  and at least as far as $z^0$. Then,  with $c_j$ as given in (\ref{E:cjsurf}), we have, as $N\to\infty$, 
\begin{equation}\label{E:constlim}
\begin{split}\frac{  c_{d-k-m}}{a_{{z^{0,N}}}^{k}c_{d-m}} &\sim
\frac{  c_{N-1-k-m}}{N^{\frac{k}{2} }c_{N-1-m}}  =\frac{\pi^{\frac{N-k-m}{2}} }{ N^{\frac{k}{2}} \Gamma\left(\frac{N-k-m}{2}  \right) }\frac{ \Gamma\left(\frac{N- m}{2}\right) }{  \pi^{\frac{N- m}{2}}    } \\
&= \pi^{-k/2} \frac{ \Gamma\left( \frac{N-m}{2}\right) }{ N^{\frac{k}{2} }\Gamma\left( \frac{N-k-m}{2}\right) } \\
&\sim  \pi^{-k/2} \frac{1}{ N^{\frac{k}{2} }  }\left(\frac{N-k-m}{2}\right)^{k/2}\\
&\sim  (2\pi)^{-k/2}.
\end{split}
\end{equation}

\subsection{The large-$N$ limit of the determinant} 
 
We use the notation $X\ominus Y$ to mean the orthogonal complement of $Y\cap X$ within $X$, where $X$ and $Y$ are closed subspaces of a Hilbert space. We have seen the following result earlier in (\ref{E:L0L0star1}) in the context of $H=l^2$.
 \begin{prop}\label{P:L0L0star}
 Let $H$, $W$, and $X$ be    Hilbert spaces, and $Q:H\to W$  and ${\cl}:H\to X$  be continuous linear mappings. Let $L$ be the restriction of $\cl$ to $\ker Q$ and $L_0$ the restriction of $L$  to the orthogonal complement of $\ker L$ within $\ker Q$:
\begin{equation}\label{E:defLL0}
L={\cl}|\ker Q\qquad\hbox{and}\qquad L_0={\cl}|(\ker Q\ominus\ker\cl).
\end{equation}
 Then 
  \begin{equation}\label{E:L0L0}
  L_0L_0^*=   {\cl}P_{\ker Q}{\cl}^*,
  \end{equation}
  where $ P_{\ker Q}$  is   the orthogonal projection in $H$ onto the subspace $\ker Q$. Here the adjoint $L_0^*$ has domain $X$ and codomain $\ker Q\ominus\ker\cl$.
  \end{prop}
  \begin{proof} Let
  $$J:D(L_0)\to H$$
  be the inclusion map, where 
  $$D(L_0)=\ker Q\ominus\ker\cl$$
  is the domain of $L_0$. The adjoint $J^*$ is the orthogonal projection of $H$ onto the subspace $D(L_0)$. Then
  $$JJ^*:H\to H$$
  is the orthogonal projection in $H$ with image being the subspace $D(L_0)$. Next we note that
  \begin{equation}\label{E:L0D0}
  L_0={\cl}J,
  \end{equation}
  and so
  \begin{equation}\label{E:JJstar}
  L_0L_0^* = {\cl} P_{\ker Q \ominus \ker\cl}{\cl}^*.
  \end{equation}
   Now for any $x\in X$ the element $ P_{\ker Q}{\cl}^*x$ is in $\ker Q$ and is orthogonal to all $v\in  \ker {\cl} \cap \ker Q$ because 
  \begin{equation}
  \la  P_{\ker Q}{\cl}^*x, v \ra =\la x, {\cl}P_{\ker Q}v\ra = \la x, {\cl}v\ra=\la x, 0\ra =0.
  \end{equation}
  Thus
  \begin{equation}
  P_{\ker Q}{\cl}^*x\in \ker Q \ominus \ker \cl.
  \end{equation}
  In general if the projection of a vector $v$ onto a subspace $X$ actually lies in a subspace $Y\subset X$ then $P_Xv$, being the point on $X$ closest to $v$ is also the point on $Y$ closest to $v$, and so $P_Xv=P_Yv$. Thus
  \begin{equation}
  P_{\ker Q}{\cl}^*x=P_{\ker Q\ominus \cl}\cl^*x.
  \end{equation}
  Hence, using (\ref{E:JJstar}), we have
   \begin{equation}\label{E:L0L0star}
  L_0L_0^* = {\cl} P_{\ker Q}{\cl}^*.
  \end{equation}
\end{proof} 
      
      We can now determine the limit of $ L_{0,N}L_{0,N}^*$ as $N\to\infty$.

\begin{prop}\label{P:L0N}  Let $H$ be a   Hilbert space, and $Z_1\subset Z_2\subset \ldots$ a sequence of finite-dimensional subspaces of $H$ whose union is dense in $H$. Let $Q:H\to W$  and ${\cl}:H\to X$ be   surjective continuous linear functions, where $X$ and $W$ are finite-dimensional Hilbert spaces, and ${\cl}_N$ and $Q_N$ their restrictions to $Z_N$:
$$Q_N=Q|Z_N\qquad\hbox{and}\qquad \cl_N=\cl|Z_N.$$
Suppose ${\cl}$ maps $\ker Q$ surjectively onto $X$.   Then:
\begin{itemize}
\item[(i)] ${\cl}_N$ maps $\ker Q_N$ surjectively onto $X$ for large $N$; \item[(ii)] the operators $L_{0,N}L_{0,N}^* $ on $W$ converge to $L_0L_0^*$:
\begin{equation}\label{E:L0N}
\lim_{N\to\infty} L_{0,N}L_{0,N}^*  =  L_0L_0^*,
\end{equation}  
where $L_0$ is the restriction of $\cl$ to  $\ker Q\ominus \ker\cl$, the orthogonal complement of $\ker Q\cap \ker {\cl}$ within $\ker Q$,  and $L_{0,N}$ is the restriction $\cl$ to $\ker Q_N\ominus \ker {\cl}_N$.
\end{itemize}
\end{prop}
Let us note what $L_{0,N}$ is more explicitly:
\begin{equation}
L_{0,N}:  \ker Q_N\cap [\ker Q_N\cap (\ker\cl_N)^\perp]\to X: z\mapsto {\cl_N}z={\cl}z.
\end{equation}
The statement that  ${\cl}_N$ maps $\ker Q_N=Z_N\cap\ker Q$ surjectively onto $X$ for large $N$ therefore means that {\em $L_{0,N}$ is an isomorphism for large $N$.}

We note that, as a consequence of (\ref{E:L0N}),
\begin{equation}\label{E:L0Ndet}
\lim_{N\to\infty} \det(L_{0,N}L_{0,N}^*)  =  \det(L_0L_0^*).
\end{equation} 
We will need this in working out the limit of the right hand side in the  disintegration formula (\ref{E:disintgenslice5}). 
 
\begin{proof} By Proposition \ref{P:L0L0star} we have:
\begin{equation}\label{E:L0lim}
 L_{0,N}L_{0,N}^*={\cl}P_{\ker Q_N}{\cl}^*\qquad\hbox{and}\qquad  L_{0}L_{0}^*={\cl}P_{\ker Q}{\cl}^*.
\end{equation}
Now for any $v\in H$ we have
\begin{equation}
\begin{split}
\lim_{N\to\infty} L_{0,N}L_{0,N}^*v &=\lim_{N\to\infty} {\cl}(P_{\ker Q_N}{\cl}^*v)\\
&= {\cl}\bigl(\lim_{N\to\infty}P_{\ker Q_N}({\cl}^*v) \bigr)\\
&={\cl}P_{\ker Q}({\cl}^*v) \qquad\hbox{by Theorem \ref{T:limQN},}\\
&=  L_{0}L_{0}^*v \qquad\hbox{by (\ref{E:L0lim}).}
\end{split}
\end{equation}
Since $X$ is finite-dimensional this pointwise convergence implies the convergence of operators which proves (\ref{E:L0N}).

 The conclusion (i) follows from the transversality result in Lemma \ref{L:transverse} and the limit result in Theorem \ref{T:limHZN3}, both proved in section \ref{s:linlim}.
\end{proof}

\subsection{The large-$N$ limit of the integrand}   Let $N_0$ be any value of $N$ for which $L_{0,N}$ is surjective onto $X$. Then let  us also recall that, for $N>N_0$, the integrand $I_N$ given in (\ref{E:IN2}):
\begin{equation}\label{E:disintI2}
 I_N(x)=  \phi(x) 
  \left\{	1 - a_{{z^{0,N}}}^{-2} \norm{{L_{0,N} }^{-1}(x-{z^{0,N}}_{(k)})}^2  \right\}^{\frac{d-k-m-1}{2}}1_{D_N}(x),\end{equation}
  where $D_N$ is the set of all $x\in\mbr^k$ for which the term within $\{\ldots\}$ is positive.
 
  Then we observe that
  \begin{equation}
  \begin{split}
   \norm{{L_{0,N} }^{-1}(x-{z^{0,N}}_{(k)})}^2  &\leq \norm{{(L_{0,N}L_{0,N}^*) }^{-1} } \norm{x-{z^{0,N}}_{(k)}  }^2\\
   &\leq  C (\norm{x}^2+\norm{  {z^{0,N}}_{(k)}  }^2)\\
   &\leq C(\norm{x}^2+ \norm{z^{0,N }_{(k)}}^2)\\
   &\leq C(\norm{x}^2+ \norm{z^{0,N } }^2),
  \end{split}
  \end{equation}
  where
  \begin{equation}
  C=\sup_{N\geq 1} 2\norm{{(L_{0,N}L_{0,N}^*) }^{-1} }<\infty,
  \end{equation}
  because of the finiteness of the limit of $\norm{{(L_{0,N}L_{0,N}^*) }^{-1} }$ as $N\to\infty$, as seen in (\ref{E:L0N}).
  
Moreover,
  $$a_{z^{0,N}}^2=N-\norm{z^{0,N}}^2$$
  lies between $N-\norm{z^{0,N_0}}^2$ and $N-\norm{z^0}^2$, because, from the definition of $z^{0,N}$ as the point on $Q^{-1}(w^0)\cap Z_N$ closest to the origin we have,
  \begin{equation}
  \norm{z^{0,N_0}}\leq \norm{z^{0,N}}\leq \norm{z^0}.
  \end{equation}
  Consequently,
  \begin{equation}
 \lim_{N\to\infty} a_{{z^{0,N}}}^{-2} \norm{{L_{0,N} }^{-1}(x-{z^{0,N}}_{(k)})}^2  =0.
 \end{equation}
 Hence,  any given point  $x\in\mbr^k$ lies in $D_N$ for $N$ large enough.

As we have just seen, the term within $\{\ldots\}$ in $I_N$ goes to $1$ as $N\to\infty$; this implies
  \begin{equation}\label{E:aNlim}
  \begin{split}
&  \lim_{N\to\infty} \left\{	1 - a_{{z^{0,N}}}^{-2} \norm{{L_{0,N} }^{-1}(x-{z^{0,N}}_{(k)})}^2  \right\}^{\frac{d-k-m-1}{2}}\\
&= \lim_{N\to\infty}  \left\{	1 - a_{{z^{0,N}}}^{-2} \norm{{L_{0,N} }^{-1}(x-{z^{0,N}}_{(k)})}^2  \right\}^{\frac{N}{2}}.\end{split}
  \end{equation}

  To work out the limit on the right side of  (\ref{E:aNlim}), let us note first that, by dominated convergence,
  \begin{equation}\label{E:explim}
  \begin{split}
(1+x_N)^{ N} &= 1+ Nx_N +\frac{ N(N-1)}{2!}x_N^2+\ldots\\
&= 1+Nx_N +\frac{1*(1-N^{-1})}{2!}(Nx_N)^2+\ldots\\
&\to \exp\left(\lim_{N\to\infty}Nx_N\right), \quad\hbox{if $\lim_{N\to\infty}Nx_N$ exists and is finite.}
\end{split}
  \end{equation}
In the present context
\begin{equation}
\begin{split}
Nx_N &= -Na_{{z^{0,N}}}^{-2} \norm{{L_{0,N} }^{-1}(x-{z^{0,N}}_{(k)})}^2.
\end{split}
\end{equation}
For $N$ large enough, independent of $x$, this is bounded above by
\begin{equation}
2 C(\norm{x}^2+ \norm{z^{0,N_0} }^2)
\end{equation}
because
$$Na_{{z^{0,N}}}^{-2} \leq N/(N-\norm{z^{0,N_0}}^2).$$
  
  Moreover,
  \begin{equation}\label{E:limNxN}
  \begin{split}
 & \lim_{N\to\infty} Nx_N \\&= -1*\lim_{N\to\infty}\la (L_{0,N}L_{0,N}^*)^{-1}(x-{z^{0,N}}_{(k)}), (x-{z^{0,N}}_{(k)})\ra\\
   &=- \la (L_0L_0^*)^{-1}(x-z^0_{(k)}), x-z^0_{(k)}\ra
  \end{split}
  \end{equation}
   where we have used the limiting formula (\ref{E:L0N}) as well as Proposition  \ref{P:limz0NK} (which implies that ${z^{0,N}}_{(k)}\to {z^{0 }}_{(k)}$ as $N\to\infty$).
 
 Returning to (\ref{E:aNlim})  we have
\begin{equation}
\begin{split}
&  \lim_{N\to\infty}  \left\{	1 - a_{{z^{0,N}}}^{-2} \norm{{L_{0,N} }^{-1}(x-{z^{0,N}}_{(k)})}^2  \right\}^{\frac{N}{2}} 1_{D_N}(x)\\
&
\hskip 1in = \exp\left(-\frac{1}{2}  \la (L_0L_0^*)^{-1}(x-z^0_{(k)}), x-z^0_{(k)}\ra\right).
 \end{split}
\end{equation}

  \subsection{Proof of Theorem \ref{T:limintfsig} for $\phi\in L^1(\mbr^k)$} Looking back at (\ref{E:disintgenslice4b}) we have
\begin{equation}\label{E:disintgenslice5}
 \lim_{N\to\infty} \int_{S_{Z_{N}}(a)\cap Q_N^{-1}(w^0)}f\,d{\bar\sigma}\\
 =  \lim_{N\to\infty}  \frac{  c_{d-k-m}}{a_{{z^{0,N}}}^{k}c_{d-m}} \int_{\mbr^k}
I_N\,\frac{dx}{|\det L_{0,N} |},
\end{equation}
where
\begin{equation}\label{E:IN3}
I_N=  \phi(x) 
  \left\{	1 - a_{{z^{0,N}}}^{-2} \norm{{L_{0,N} }^{-1}(x-{z^{0,N}}_{(k)})}^2  \right\}^{\frac{d-k-m-1}{2}}1_{D_N}(x).\end{equation}
We have already determined the limits of the constant term outside the integral  (in (\ref{E:constlim})), as well as those of the full integrand on the right hand side. Moreover, we observe that 
$$|I_N|\leq |\phi(x)|.$$
Thus, assuming that $\phi$ is integrable over $\mbr^k$, we can apply dominated convergence to conclude that
\begin{equation}\label{E:disintgenslice6}
\begin{split}
& \lim_{N\to\infty} \int_{S_{Z_{N}}(a)\cap Q_N^{-1}(w^0)}f\,d{\bar\sigma}\\
& =  (2\pi)^{-k/2}\int_{\mbr^k}\phi(x) \exp \left({-\frac{1}{2}\la  (L_0L_0^*)^{-1}(x-z^0_{(k)}), x-z^0_{(k)} \ra }\right) \,\frac{dx}{\sqrt{ \det(L_{0}L_{0}^*).  }  }
\end{split}
\end{equation}

  \subsection{Proof of Theorem \ref{T:limintfsig} for  more general $\phi$} 
  \begin{proof}  Let $\phi$ be any bounded Borel function on $\mbr^k$. Let us recall from (\ref{E:disintgenslice4b}) that:
  \begin{equation}\label{E:disintgenslice4c}
 \int_{S_{Z_N}(a)\cap Q_N^{-1}(w^0)}f\,d{\bar\sigma} =
\int_{\mbr^k}\phi(x)\,d\mu_N(x),
\end{equation}
where 
  $$f(x)= \phi(x_1,\ldots, x_k)\qquad\hbox{for all $x=(x_1,x_2,\ldots)$}$$
and
\begin{equation}\label{E:dmuN}
\begin{split}
&d\mu_N(x)\\
&=  
 \frac{  c_{d-k-m}}{a_{{z^{0,N}}}^{k}c_{d-m}}  \left\{	1 - a_{{z^{0,N}}}^{-2} \norm{{L_{0,N} }^{-1}(x-{z^{0,N}}_{(k)})}^2  \right\}^{\frac{d-k-m-1}{2}}1_{D_N}(x)\, \frac{dx}{|\det L_{0,N} |}.
 \end{split}
 \end{equation}
 Taking $\phi=1$ in (\ref{E:disintgenslice4c}) we see that $\mu_N$ is a probability measure. 
 Now let $\mu_{\infty}$ be the Gaussian measure on $\mbr^k$ given by
 \begin{equation}\label{E:muinfty}
 d\mu_\infty(x)= (2\pi)^{-k/2} \exp\left({-\frac{1}{2}\la  (L_0L_0^*)^{-1}(x-z^0_{(k)}), x-z^0_{(k)} \ra }\right) \,\frac{dx}{\sqrt{ \det(L_{0}L_{0}^*)  }  }.
 \end{equation}
With this notation, the result (\ref{E:disintgenslice6}) says that
 \begin{equation}\label{E:disintgenslice7}
 \lim_{N\to\infty} \int_{\mbr^k}\psi\,d\mu_N=\int_{\mbr^k}\psi\,d\mu_{\infty}\qquad\hbox{for all $\psi\in L^1(\mbr^k)$.}
 \end{equation}
 Taking $\psi$  to be the indicator function of any compact set $B$ we have:
 \begin{equation}\label{E:disintgenslice7b}
 \lim_{N\to\infty}\mu_N(B) =\mu_{\infty}(B).
 \end{equation}
Since $\mu_N$ and $\mu_{\infty}$ are probability measures, this also implies
 \begin{equation}\label{E:disintgenslice8}
 \lim_{N\to\infty}\mu_N(B^c) =\mu_{\infty}(B^c).
 \end{equation}
 Now let
  $\epsilon>0$. Then there is a compact set $B_\epsilon\subset \mbr^k$ for which 
 $$\mu_{\infty}(B_{\epsilon})>1-\epsilon.$$
 We have
 \begin{equation}\label{E:limN}
 \begin{split}
 \int_{\mbr^k} \phi  \,d\mu_N - \int_{\mbr^k} \phi  \,d\mu_{\infty} &= \int_{\mbr^k} \phi 1_{B_{\epsilon}}\,d\mu_N-  \int_{\mbr^k} \phi 1_{B_{\epsilon}} \,d\mu_\infty \\
 &\qquad\qquad + \int_{\mbr^k} \phi 1_{B_{\epsilon}^c } \,d\mu_N- \int_{\mbr^k} \phi  1_{B_{\epsilon}^c }\,d\mu_\infty.\\
 &
 \end{split}
 \end{equation}
Taking $\psi$ to be $\phi 1_{B_{\epsilon}}$, which is integrable over $\mbr^k$, in (\ref{E:disintgenslice7}),  we have
\begin{equation}
\lim_{N\to\infty}\left[\int_{\mbr^k} \phi 1_{B_{\epsilon}}\,d\mu_N-  \int_{\mbr^k} \phi 1_{B_{\epsilon}} \,d\mu_\infty\right]=0.
\end{equation}
Next,
\begin{equation}
\begin{split}
\limsup_{N\to\infty} \Big|\int_{\mbr^k} \phi 1_{B_{\epsilon}^c } \,d\mu_N\Big| &\leq\norm{\phi}_{\sup}\lim_{N\to\infty}\mu_N(B_{\epsilon}^c) \\
&= \norm{\phi}_{\sup} \mu_{\infty}(B_{\epsilon}^c)\\
&\leq  \norm{\phi}_{\sup}\epsilon.
\end{split}
\end{equation}
Using these observations  in (\ref{E:limN}) we have
\begin{equation}
\limsup_{N\to\infty}\Big|\int_{\mbr^k} \phi  \,d\mu_N - \int_{\mbr^k} \phi  \,d\mu_{\infty}\Big| \leq 0+2\norm{\phi}_{\sup}\epsilon.
\end{equation}
Since $\epsilon>0$ is arbitrary, this establishes our goal:
\begin{equation}\label{E:philimgoal}
\lim_{N\to\infty}\int_{\mbr^k}\phi\,d\mu_N= \int_{\mbr^k}\phi\,d\mu_\infty.
\end{equation}
This completes the proof of  Theorem \ref{T:limintfsig}. 
\end{proof}

 \subsection{Proof of Theorem \ref{T:RadnonNlim}}
   Let $\mu$ be the probability measure on $\mbr^\infty$  (the space of all real sequences) that is specified by the characteristic function
 \begin{equation}\label{E:mucf}
 \int_{\mbr^\infty}\exp\left({i\la t, x\ra}\right)\,d\mu(x)=\exp\left({i\la t, z^0\ra -\frac{1}{2}\norm{P_0t}^2}\right),
 \end{equation}
 where $t\in l^2$ is any sequence with finitely many nonzero entries,  $P_0$ is the orthogonal projection in $l^2$ onto a closed subspace of finite codimension $m$, and $z^0$ is any point in $l^2$.
 Let us determine the pushforward measure ${\pi_{(k)}}_*\mu$ of $\mu$ to $\mbr^k$:
 \begin{equation}
{ \pi_{(k)}}_*\mu(S)=\mu\bigl(\pi_{(k)}^{-1}(S)\bigr),\qquad\hbox{for all Borel $S\subset\mbr^k$,}
 \end{equation}
 where 
 $$\pi_{(k)}:\mbr^\infty\to\mbr^k:z\mapsto z_{(k)}=(z_1,\ldots, z_k)$$
  is the projection on the first $k$ coordinates.
 
   Let  $\cl$ be the restriction of $\pi_{(k)}$ to $l^2$:
   \begin{equation}
   \cl:l^2\to X=\mbr^k: z\mapsto z_{(k)}.
   \end{equation}
    Then the adjoint is
$$\cl^*: X\to l^2: v\mapsto (v,0,0,\ldots).$$
 The image of the orthogonal projection $P_0$ is of the form
 $$Q^{-1}(w^0),$$
 where $Q:l^2\to W=\mbr^m$ is a continuous linear sujection and $w^0$ is a point in $W$. Thus
 $$P_0= P_{\ker Q}.$$
 Then for any $t\in\mbr^k$, we have
 \begin{equation}\label{E:cfpikmu}
\begin{split}
  \int_{\mbr^k}\exp\left(i\la t,  x\ra\right)\,d \pi_{{(k)}_*}\mu(x)&=
 \int_{\mbr^\infty}\exp\left(i\la t, \pi_{(k)}x\ra\right)\,d \mu(x) \\
  &= \int_{\mbr^\infty}\exp\left(i\la {\cl}^*t,  x\ra\right)\,d \mu(x) \\
 &=\exp\left(i\la {\cl}^*t, z^0\ra -\frac{1}{2}\norm{P_0{\cl}^*t}^2\right)\\
 &= \exp\left(i\la   t, {\cl}z^0\ra -\frac{1}{2}\la P_{\ker Q}{\cl}^*t, P_{\ker Q}{\cl}^*t\ra \right)\\
 &= \exp\left(i\la  t, {\cl}z^0\ra -\frac{1}{2}\la {\cl} P_{\ker Q}{\cl}^*t, t\ra \right).
 \end{split}
 \end{equation}

 The measure $\mu_{\infty}$  in  (\ref{E:muinfty}) is given by
 \begin{equation}\label{E:muinfty2}
 d\mu_\infty(x)= (2\pi)^{-k/2} \exp\left({-\frac{1}{2}\la  (L_0L_0^*)^{-1}(x-z^0_{(k)}), x-z^0_{(k)} \ra }\right) \,\frac{dx}{\sqrt{ \det(L_{0}L_{0}^*)  }  },
 \end{equation}
 where $L_0:\ker Q\to \mbr^k$ is the restriction of $\cl$ to the Hilbert space $\ker Q\subset l^2$.
 Its characteristic function is given by
 \begin{equation}\label{E:cfmuinfty}
 \begin{split}
& \int_{\mbr^k}\mbox{exp}\left(i\la t, x\ra\right) \,d\mu_{\infty}(x) \\
&=\int_{\mbr^k} \mbox{exp}\left(i\la (L_0L_0^*)^{1/2}t,y\ra +i\la t, z^0_{(k)}\ra (2\pi)^{-k/2}-\frac{\norm{y}^2}{2} \right)\,dy\\
 &=\mbox{exp}\left(i\la t, z^0_{(k)}\ra -\frac{1}{2}\norm{ (L_0L_0^*)^{1/2}t}^2\right)
 \end{split}
 \end{equation}
 where, in the first line,  we used the natural change of variables $x=(L_0L_0^*)^{1/2}y+z^0_{(k)}$, and for the second line we used a standard formula for Gaussian integration.  Now we recall from (\ref{E:JJstar2}) that $L_0L_0^*$ equals ${\cl}P_{\ker Q}\cl^*$. Thus,
\begin{equation}\label{E:cfmuinfty2}
 \begin{split}
 \int_{\mbr^k}\exp\left({i\la t, x\ra}\right) \,d\mu_{\infty}(x) &= \exp\left({i\la t, {\cl}z^0\ra -\frac{1}{2}\la {\cl}P_{\ker Q}\cl^*t,t\ra}\right).
 \end{split}
 \end{equation}
 This agrees exactly with  the characteristic function for ${\pi_{(k)}}_*\mu$ we obtained in (\ref{E:cfpikmu}). Hence
 \begin{equation}
 {\pi_{(k)}}_*\mu =\mu_{\infty}.
 \end{equation}
Combining this with the result of  Theorem \ref{T:limintfsig} given in (\ref{E:philimgoal}), we conclude that
\begin{equation}\label{E:mainresu1}
\begin{split}
&\lim_{N\to\infty} \int_{S_{Z_N}(a)\cap Q_N^{-1}(w^0)}\phi(x_1,\ldots, x_k)\,d{\bar\sigma}(x_1,\ldots, x_N) \\
&=   \int_{\mbr^k} \phi \,d \mu_{\infty}\\
&= \int_{\mbr^k} \phi \,d {\pi_{(k)}}_*\mu = \int_{\mbr^\infty} \phi\circ{\pi_{(k)}} \,d\mu.
\end{split}
\end{equation}
This completes the proof of Theorem \ref{T:RadnonNlim}.  
   
 \section{The result in Abstract Wiener Spaces}\label{s:aws}

The concept of an Abstract Wiener Space was introduced by L. Gross \cite{GrossAWS} and is a standard framework within which Gaussian measures on infinite dimensional spaces are studied. Let $H$ be an infinite dimensional real separable Hilbert space. We work with a {\em measurable norm } $|\cdot|$  on $H$; this is a norm with the property that for any $\epsilon>0$ there is a finite-dimensional subspace $F_{\epsilon}$ of $H$ such that for any finite-dimensional subspace $F'$ of $H$ orthogonal to $F_{\epsilon}$ we have
$$\gamma_{F'}\{x\in F'\,:\, |x|>\epsilon\}<\epsilon,$$
where $\gamma_{F'}$ is the standard Gaussian measure on $F'$. Let $B$ be the Banach space obtained by completion of $H$ with respect to $|\cdot|$.   The natural injection
$$j:H\to B$$
is continuous with $H$ having the Hilbert-space topology (see, for example,  Eldredge \cite[page 16]{Eld2016}). Then $\phi\in B^*$  restricts to a continuous linear functional  on the Hilbert space $H$, and so is given by
\begin{equation}\label{E:phi0}
\phi\bigl(j(x)\bigr)=\la j^*\phi,x\ra \qquad\hbox{for all $x\in H$}
\end{equation}
for a unique element $j^*\phi\in H$. Thus we have a continuous linear injection
$$j^*:B^*\to H,$$
and the image of $j^*$ is a dense subspace of $H$.
 
With notation as above, let $L$ be a closed affine subspace of $H$. Then, as shown in \cite{HolSenGR2012}, there is a  Borel measure $\mu_L$ on $B$  such that every $\phi\in B^*$, viewed as a random variable defined on $B$, has Gaussian distribution specified by
\begin{equation}\label{E:IL1}
\int_B\exp\left(it \phi \right)\,d\mu_L = \exp\left(it\la p_L,j^*\phi\ra -\frac{t^2}{2}\norm{P_0(j^*\phi)}^2\right)\quad\hbox{for all $t\in\mbr$,}
\end{equation}
where $p_L$ is the point on $L$ closest to $0$ and $P_0:H\to H$ is the orthogonal projection onto the subspace
$$L_0=L-p_L.$$
The linear mapping
\begin{equation}
j^*(B^*)\to L^2(B,\mu_L): j^*(\phi)\mapsto \phi,
\end{equation}
  extends to a continuous linear mapping
\begin{equation}\label{E:IL}
I_L:H\to L^2(B, \mu_L).
\end{equation}
Moreover,
\begin{equation}\label{E:ILB}
\int_B\exp\left(iI_L(h)\right)\,d\mu_L = \exp\left(i\la p_L,h\ra -\frac{1}{2}\norm{P_0h}^2\right)\quad\hbox{for all $h\in H$.}
\end{equation}
To compare with a familiar situation we observe that if $H$ is finite-dimensional then  $B=H$, and:
\begin{equation}\label{E:ILfin}
\hbox{$I_L(h)  =\la \cdot, h\ra_H$.}
\end{equation}

The Gaussian measure $\mu_L$   is supported on the closure $\overline{L}$ of $j(L)$ inside $B$.

Now let us see how one can extend  a function from a finite-dimensional subspace $V$ of $H$ to a function on the Banach space $B$. First let us note that if $W$ is a closed subspace of $H$ that contains $V$ then we have the function $f_W$ on $W$ given by
\begin{equation}
f_W =f_V\circ P^W_V,
\end{equation}
where
$$P^W_V:W\to V$$
is the orthogonal projection onto the subspace $V\subset  W$.  Suppose $h_1,\ldots, h_k$ is an orthonormal basis  of $V$. Then
\begin{equation}\label{E:PWV}
P^W_V(w)=\sum_{r=1}^k \la w, h_r\ra h_r \qquad\hbox{for all $w\in W$.}
\end{equation}

Next we extend this process all the way to $B$. However, there is no ``orthogonal projection'' from $B$ onto $V$. Nonetheless, by choosing an orthonormal basis $h_1,\ldots, h_k$ of $V$ we can define
\begin{equation}\label{E:PBV}
P^B_V =\sum_{r=1}^kI_L(h_r)h_r,
\end{equation}
where $I_L$ is as in (\ref{E:IL}). If $H$ is  finite-dimensional then, in view of (\ref{E:ILfin}),   we can see that the expression for $P^B_V$ in (\ref{E:PBV}) agrees with (\ref{E:PWV}). If $V$ happens to be contained in the subspace $j^*(B^*)$ then $P^B_V$ is given more clearly by
$$P^B_Vx= \sum_{r=1}^k \phi_r(x)h_r$$
where $\phi_r$ is the point in $ B^*$ for which $j^*(\phi_r)=h_r$. If $f$ is a function on $V$ then we can extend to a $\mu_L$-almost-everywhere defined function $f_B$ on $B$ by:
\begin{equation}
f_B=f\circ P^B_V.
\end{equation}
(This notion was discussed in Gross \cite{GrossAWS} for $V\subset j^*(B^*)$.)
With this notation, we can formulate our main result in the setting of Abstract Wiener Spaces.

\begin{theorem}\label{T:mainAWS}
Let $H$ be a real separable Hilbert space and $B$ the closure of $H$ with respect to a measurable norm. Let $L$ be a closed affine subspace of  $H$ of finite codimension.  Let $f$ be a Borel function on a finite-dimensional nonzero subspace $V$ of $H$ such that the orthogonal projection $P^H_V$ maps $L$ onto $V$. Suppose $Z_1, Z_2, \ldots$ are  finite-dimensional subspaces of $H$, and
$$V\subset Z_1\subset Z_2\subset\ldots \subset H$$
with $\cup_{N\geq 1}Z_N$ being dense in $H$. Then
\begin{equation}\label{E:AWSlim}
\lim_{N\to\infty}(R_{Z_N}f_{Z_N})(L\cap S_{Z_N})=Gf_B(L)
\end{equation}
Here,  on the left is the normalized  surface-area integral of $f_{Z_N}$ over the circle $L\cap S_{Z_N}$ formed by intersecting $L$ with the sphere in $Z_N$ of radius  $\sqrt{\dim Z_N}$, and on the right is the integral of $f_B$ over $B$ with respect to the measure $\mu_L$.
\end{theorem}

\begin{proof}
Let
$$L_0=L-p_L,$$
be the subspace of $H$ parallel to $L$; here $p_L$ is the point on $L$ closest to $0$. 
Since the affine subspace $L$ is of finite codimension, there is an orthonormal basis $u_1,\ldots, u_m$ of   $L_0^\perp$.  Let
$$\hbox{$u_{j,N}$ be the orthogonal projection of $u_j$ onto $Z_N$.}$$
Since $\cup_{N\geq 1}Z_N$ is dense in $H$, $u_{j,N}\neq 0$ for large $N$  (Lemma \ref{L:indep}). Thus,
  for $N$ large enough, the orthogonal projection $u_{j,N}\neq 0$ for every $j\in\{1,\ldots, m\}$.

  A vector $v\in Z_N\subset H$ is orthogonal to $u_1,\ldots, u_m$ if and only if it is orthogonal to the vectors $u_{1,N},\ldots u_{m,N}$.
  
  A point $x\in H$ lies in  $L$ if and only if
  $$\la x, u_1\ra =p_1,\ldots, \la x, u_m\ra=p_m,$$
  where $p_j=\la p_L, u_j\ra$ for each $j$. Thus
 $L_N=L\cap Z_N $ consists of all points $x\in Z_N$ satisfying
 \begin{equation}\label{E:xuj}
 \la x, u_{j,N}\ra=p_j\qquad\hbox{for all $j\in\{1,\ldots, m\}$.}
 \end{equation}
 For large $N$ the set of all such $x$ constitutes an affine subspace in $Z_N$ of codimension $m$.
 
 Let $V_0$ be the orthogonal projection of $L_0$ on $V$:
 $$V_0=P^H_V(L_0).$$
  If $L_0=\{0\}$ then $P^H_V(L)$ consists of just one point, and the functions $f_{Z_n}$ and $f_B$ are all constant, equal to the value of $f$ at that point. In this case our main result is true because both sides are equal to the value of $f$ at this point. So now we assume that $V_0\neq 0$ has dimension $k\geq 1$.  Let us choose an orthonormal basis $h_1, h_2, h_3,  \ldots$ of $H$ such that the first $k$ vectors $h_1,\ldots, h_k$ form a basis of $V_0$.

 Let
 $$d_N=\dim Z_N.$$
 Then 
 \begin{equation}
 \begin{split}
 (R_{Z_N}f_{Z_N })(S_{L_N}) &= \int_{S_{L_N}} f(x_1h_1+\ldots +x_kh_{k})\,d\ovs(x_1,\ldots, x_{d_N}),
 \end{split}
 \end{equation}
 where $S_{L_N}$ is the circle in $Z_N$ formed by the intersection of the sphere of radius $\sqrt{d_N}$ in $\mbr^{d_N}$
 with the affine subspace of $\mbr^{d_N}$ comprised of all points $x$ for which $x_1h_1+\ldots +x_{d_N}h_{d_N}$ lies in $L$.   
 
 We identify $H$ with $l^2$ via the orthonormal basis $h_1, h_2,\ldots$, and denote again by $L$ the affine subspace of $l^2$ that corresponds to $L\subset H$.

 Then by Theorem \ref{T:limintfsig}
 
  \begin{equation}\label{E:RNSZNf}
 \begin{split}
& \lim_{N\to\infty}(R_{Z_N}f_{Z_N })(S_{L_N}) \\
&= (2\pi)^{-k/2}\int_{\mbr^k}  f(x_1h_1+\ldots +x_kh_{k})\cdot\\
&\hskip 1in \cdot \exp\left(-\frac{1}{2}\la  (L_0L_0^*)^{-1}(x-z^0_{(k)}), x-z^0_{(k)} \ra \right) \,\frac{dx}{\sqrt{ \det(L_{0}L_{0}^*)  }  },
 \end{split}
 \end{equation}
with notation as in Theorem \ref{T:limintfsig}. The right hand side here is equal to
\begin{equation}\label{E:intGfBL}
\int_B f\bigl(I_L(h_1)h_1+\ldots +I_L(h_k)h_k\bigr)\,d\mu_L,
\end{equation}
which we see by observing that
\begin{equation}
\begin{split}
&
\int_B \exp\left({i\bigl(t_1I_L(h_1) +\ldots +t_kI_L(h_k)\bigr)}\right)\,d\mu_L\\
& = \int_B \exp\left({iI_L(t_1h_1+\ldots+t_k h_k)}\right)\,d\mu_L\\
&=\exp\left({ i\la z^0, t_1h_1+\ldots +t_k h_k\ra -\frac{1}{2}\norm{P_{L_0}\bigl(t_1 h_1  +\ldots +t_k  h_k \bigr)}^2  }\right)\\
&=\exp\left({i\la t, z^0_{(k)}\ra -\frac{1}{2} \la (L_0L_0^*)t,t\ra}\right).
\end{split}
\end{equation}
This implies that the distribution of $\bigl(I_L(h_1),\ldots, I_L(h_k)\bigr)$  has the Gaussian density
\begin{equation}
(2\pi)^{-k/2} \exp\left(-\frac{1}{2}\la  (L_0L_0^*)^{-1}(x-z^0_{(k)}), x-z^0_{(k)} \ra \right) \,\frac{dx}{\sqrt{ \det(L_{0}L_{0}^*)  }  }
\end{equation}
that appears on the right side in (\ref{E:RNSZNf}).  We have thus completed the proof, since (\ref{E:intGfBL}) is exactly $Gf_B(L)$.
 \end{proof}

\section{Some linear algebra  and limits}\label{s:linlim}

In this section we prove results that we have used in earlier sections. We will often use the following notation:
\begin{equation}\label{E:minus}
X\ominus Y=X\cap (X\cap Y)^\perp,
\end{equation}
which is the orthogonal complement of $X\cap Y$ within $X$, where $X$ and $Y$ are subspaces of any given inner-product space.

\subsection{Subspaces and Projections}

Let us begin with some observations about projections onto subspaces.
 
The following  result is a basic observation about  how subspaces in a vector space may be situated relative to each other.

\begin{lemma}\label{L:transverse} Let   ${\cl}:H\to X$ and $Q:H\to W$ be surjective linear maps between vector spaces. Then the following are equivalent:
\begin{itemize}
\item[(i)] ${\cl}$ maps $\ker Q$ surjectively onto $X$;
\item[(ii)] $Q$ maps $\ker {\cl}$ surjectively onto $W$;
\item[(iii)] $\ker {\cl}+\ker Q=H$.
\end{itemize}
More generally, ${\cl}$ maps a subspace $V\subset H$ onto $X$ if and only if $V+\ker \cl=H$. 
\end{lemma}
\begin{proof}  Let $V$ be a subspace of $H$. Then
$${\cl}^{-1}\bigl({\cl}(V)\bigr)=V+\ker {\cl}.$$
If $V+\ker {\cl}=H$ then  $X={\cl}(H)={\cl}(V)$. Conversely, if $\cl(V)=X$ then
$V+\ker {\cl}= {\cl}^{-1}\bigl({\cl}(V)\bigr)={\cl}^{-1}(X)=H$. 

Thus (iii) is equivalent to (i) and also (with $\cl$ and $Q$ interchanged) to (ii). 
\end{proof}

Let $H$ be a Hilbert space, $K$ and $M$ closed subspaces of $H$ such that $K^\perp \subset M$.  If  we split $v\in M$ as $P_Kv+P_{K^\perp}v$, then in this the second vector is in $M$ and hence so is the first vector; thus $P_Kv\in M$ if $v\in M$. This means that the point $P_Kv$ on $K$ closest to $v$ is in fact in $K\cap M$. Thus
\begin{equation}\label{E:PKP}
P_Kv=P_{K\cap M}v \qquad\hbox{if $v\in M$.}
\end{equation}

Here are some more   observations on how subspaces are situated relative to each other.

\begin{prop}\label{P:ImRper} Let $R$ be an orthogonal projection in a Hilbert space $H$ and $K$ a closed subspace of $H$. Then
\begin{equation}\label{E:Rorth}
 {\rm Im}(R)\cap K^\perp= {\rm Im}(R)\cap [R(K)]^\perp.
\end{equation}
Moreover, the orthogonal complement of ${\rm Im}(R)\cap K$ within ${\rm Im}(R)$ is the image under $R$ of the orthogonal complement of ${\rm Im}(R)\cap K$:
\begin{equation}\label{E:Rorth2}
{\rm Im}(R)\ominus\left( {\rm Im}(R)\cap K\right)=R\left([{\rm Im}(R)\cap K]^\perp\right).
\end{equation}
 \end{prop}
 
 Take $R=P_{Z_N}$, the orthogonal projection onto the closed subspace $Z_N$, and $K=\ker Q$, where $Q$ is any continuous linear mapping on $H$, we have 
 \begin{equation}
 Z_N\cap (\ker Q)^\perp=Z_N\cap [P_{Z_N}(\ker Q)]^\perp
 \end{equation}
 and
  \begin{equation}
 P_{Z_N}\left([\ker Q_N]^\perp\right)= Z_N\ominus \bigl(\ker Q_N\bigr),
 \end{equation}
 where $Q_N=Q|Z_N$.
 
\begin{proof}
If $v\in {\rm Im}(R)\cap K^\perp$ then $Rv=v$  and for any $z\in K$ we have
$$\la v, R(z)\ra =\la Rv, z\ra=\la v, z\ra =0,$$
and so the vector $v$ in ${\rm Im}(R)$ is orthogonal to $R(K)$. Conversely, if $v\in {\rm Im}(R)$ is orthogonal to $R(K)$ then for any $z\in K$ we have
$$\la v, z\ra =\la Rv, z\ra=\la v, Rz\ra =0,$$
which shows that $v$ is orthogonal to $K$.  This establishes the equality (\ref{E:Rorth}). 

The equality
\begin{equation}\label{E:Rvz}
\la v, Rz\ra=\la v, R^2z\ra= \la Rv, Rz\ra 
\end{equation}
holds for all $v, z\in H$.  If $Rv$ is orthogonal to ${\rm Im}(R)\cap K$ then $\la Rv, Rz\ra=0$ for all $Rz\in K$  and so, by  (\ref{E:Rvz}), $v$ is orthogonal to ${\rm Im}(R)\cap K$.  Conversely, if $v$ is orthogonal to ${\rm Im}(R)\cap K$ then the first term in   (\ref{E:Rvz}) is $0$ whenever $Rz\in K$, and so by   (\ref{E:Rvz}) it follows that $Rv$ is orthogonal to ${\rm Im}(R)\cap K$.
\end{proof}

\subsection{Limits of projections}

 Much of our work takes places in a Hilbert space $H$ equipped with an increasing sequence of  closed or finite-dimensional subspace $Z_1\subset Z_2\subset \ldots$ whose union is dense in $H$.

\begin{lemma}\label{L:incsub} Let $Z_1\subset Z_2\subset \ldots$ be a sequence of closed subspaces of a Hilbert space. Then the following are equivalent:
\begin{itemize}
\item[(i)] $\cup_{N\geq 1}Z_N$ is dense in $H$;
\item[(ii)] $\lim_{N\to\infty}P_{Z_N}z=z$ for all $z\in H$, where $P_{Z_N}$ is the orthogonal projection onto $Z_N$.
\end{itemize}
\end{lemma}
\begin{proof} Suppose (i) holds. Let $z\in H$. Then there is a sequence of points $w_n\in \cup_{N\geq 1}Z_N$ converging to $z$. For each $k$ there is an integer $N_k$ such that $w_k\in Z_{N_k}$.  Then, bearing in mind that $P_{Z_{N_k}}z$ is the point on $Z_{N_k}$ closest to $z$, we have
$$\norm{z-P_{Z_{N_k}}z}\leq \norm{z-w_k}\to 0\quad\hbox{as $k\to\infty$.}$$
So for any $\epsilon>0$ there is an integer $k$ such that 
$$\norm{z-P_{Z_{N_k}}z}<\epsilon.$$
For $N>N_k$ the subspace $Z_{N_k}$ is contained in $Z_N$ and so $P_{Z_N}z$ being the point on $Z_N$ closest to $z$, we have
$$\norm{z-P_{Z_N}z}<\epsilon$$
for all $N>N_k$. Thus (ii) holds.

The implication (ii) $\implies$ (i) holds since all the points $P_{Z_N}z$ lie in the union $\cup_{N\geq 1}Z_N$.
\end{proof}
 
The following result was needed in proving Proposition \ref{P:L0N} and is thus a crucial result for our purposes.
 
 \begin{theorem}\label{T:limQN}  Suppose $H$ is a Hilbert space,  and $Z_1\subset Z_2\subset \ldots$ is a sequence of finite dimensional subspaces whose union is dense in $H$. Let   $L_0$ be a closed subspace of $H$ of finite codimension.  Then 
\begin{equation}\label{E:limQN1}
\lim_{N\to\infty}P_{ L_0\cap Z_N}z=P_{L_0}z.
\end{equation}
for all $z\in H$. In particular, if $L_0=\ker Q$, for some continuous linear mapping $Q:H\to W$ onto a finite-dimensional space, then
\begin{equation}\label{E:limQN}
\lim_{N\to\infty}P_{ \ker Q_N }z=P_{\ker Q}z \qquad\hbox{for all $v\in H$,}
\end{equation}
where $Q_N=Q|Z_N$.
\end{theorem}
Let us note that (\ref{E:limQN1}) implies that any element $z\in L_0$ is the limit of a sequence of elements $P_{L_0\cap Z_N}z\in L_0\cap Z_N$; thus
\begin{equation}\label{E:L0ZN}
\hbox{\em $\cup_{N\geq 1}(L_0\cap Z_N)$ is dense in $L_0$.}
\end{equation}
This conclusion requires that $L_0$ be of finite codimension; otherwise, we could just choose $L_0$ to be the line through any non-zero vector outside $\cup_{N\geq 1}Z_N$ and obtain a contradiction.

\begin{proof} Let $u_1,\ldots, u_m$ be an orthonormal basis of $L_0^\perp$.  Then for any $v\in Z_N$ we have
$$\la P_{Z_N}u_a, v\ra=\la u_a, P_{Z_N}v\ra =\la u_a, v\ra.$$
Thus if $v\in Z_N$ is orthogonal to  $[P_{Z_N}u_1,\ldots, P_{Z_N}u_m]$  then $v\in [u_1,\ldots, u_m]^\perp$ and so $v\in L_0$.  Conversely, if $v\in Z_N$ is also in $L_0$ then $v$ is orthogonal to each $u_a$ and hence to each $P_{Z_N}u_a$. Thus the orthogonal complement of $[P_{Z_N}u_1,\ldots, P_{Z_N}u_m]$ in $Z_N$ is   $L_0\cap Z_N$:
\begin{equation}
Z_N\ominus P_{Z_N}(L_0^\perp)=Z_N\cap L_0.
\end{equation}
Consequently,
\begin{equation}\label{E:ZNL0p}
Z_N\ominus L_0=P_{Z_N}(L_0^\perp).
\end{equation}

By Lemma \ref{L:limproj} (proven below) the orthogonal projection   in $H$ onto $P_{Z_N}(L_0^\perp)$  converges pointwise, as $N\to\infty$, to the orthogonal projection onto  $L_0^\perp$. Thus, using (\ref{E:ZNL0p}),
\begin{equation}\label{E:limZN}
\lim_{N\to\infty} P_{Z_N\ominus L_0}z = P_{L_0^\perp}z \qquad\hbox{for all $z\in H$.}
\end{equation}
 Since $H$ is the sum of the mutually orthogonal subspaces $Z_N\ominus L_0$, $Z_N\cap L_0$, and $Z_N^\perp$, we have
$$z= P_{Z_N\ominus L_0}z+ P_{Z_N\cap L_0}z+P_{Z_N^\perp}z,$$
and so 
\begin{equation}
 P_{Z_N\cap L_0}z= z-P_{Z_N^\perp}z- P_{Z_N\ominus L_0}z=  P_{Z_N}z - P_{Z_N\ominus L_0}z,
 \end{equation}
for all $z\in H$. Then
\begin{equation}\label{E:limZN2}
\begin{split}
\lim_{N\to\infty} P_{  L_0\cap Z_N}z &= \lim_{N\to\infty} (P_{Z_N}z-P_{  Z_N\ominus L_0}z) \\
&=z- P_{L_0^\perp}z\\
&= P_{L_0}z \\
\end{split}
\end{equation}\end{proof}

We prove the lemma used in the preceding proof.
 
\begin{lemma}\label{L:limproj}
Let $H$ be a Hilbert space and $K$   a finite-dimensional subspace  of $H$.  Suppose that $R_1, R_2, \ldots$ are   orthogonal projections in $H$ such that
$$R_Nz\to z\qquad\hbox{for all $z\in H$, as $N\to\infty$.}$$
Let $S_N$  be the orthogonal projection in $H$ onto $R_N(K)$. Then $S_N$ converges pointwise to the orthogonal projection onto $K$:
$$\lim_{N\to\infty}S_Nz=Sz\qquad\hbox{for all $z\in H$},$$
where $S$ is the orthogonal projection onto $K$.
\end{lemma}
In this result the hypothesis that $K$ is finite-dimensional is needed. For, consider $H=l^2$, $R_N$ the orthogonal projection given by $R(x_1, x_2,\ldots)=(x_1,\ldots, x_N, 0,0,\ldots)$, and $K=v^\perp$, where $v$ is the vector $(1,1/2,1/3,\ldots)$. Then $R_N(K)=\mbr^N\times\{(0,0,\ldots)\}$, and $S_N=R_N$. The pointwise limit of $S_N$ is $I$, which is not the same as the orthogonal projection onto $K$.

\begin{proof} Let $u_1,\ldots, u_m$ be an  orthonormal basis of $K=S(H)$ (if $S=0$ the result is obvious). By Lemma \ref{L:indep} (below) we may assume that $N$ is large enough that the vectors $R_Nu_1,\ldots, R_Nu_m$ are linearly independent, and thus form a basis of $R_N(K)$.  Since $S_N$ is the orthogonal projection onto this subspace, for any $z\in H$ the vector $S_Nz$  can be expressed in terms of $R_Nu_1,\ldots, R_Nu_m$ as follows:
\begin{equation}\label{E:SNz}
S_Nz=\sum_{i=1}^m c_i R_Nu_i.
\end{equation}
The coefficients $c_i$ depend on $N$. 
Since $S_Nz$ is the orthogonal projection of $z$ on the span of $\{R_Nu_1,\ldots, R_Nu_m\}$, the inner product of $S_Nz$ with each $R_Nu_j$ is the same as the inner product of $z$ with $R_Nu_j$:
\begin{equation}\label{E:zRNu}
\la z, R_Nu_j\ra =\la S_Nz, R_Nu_j\ra =\sum_{i=1}^m \la R_Nu_j, R_Nu_i\ra c_i.
\end{equation}
Thus the vector $c(N)$ of coefficients $c_j$ is
\begin{equation}\label{E:cN}
c(N)\stackrel{\rm def}{=} \left[\begin{matrix} c_1 \\ \vdots \\ c_m\end{matrix}\right]
=[\la R_Nu_a, R_Nu_b\ra]^{-1}\left[\begin{matrix} \la z, R_Nu_1\ra \\ \vdots \\ \la z, R_Nu_m\ra\end{matrix}\right].
\end{equation}
Letting $N\to\infty$ we obtain (using continuity of matrix inversion):
\begin{equation}\label{E:cinf}
\lim_{N\to\infty}c(N)= [\la u_a, u_b\ra]^{-1}\left[\begin{matrix} \la z,u_1\ra \\ \vdots \\ \la z, u_m\ra\end{matrix}\right]= \left[\begin{matrix} \la z,u_1\ra \\ \vdots \\ \la z, u_m\ra\end{matrix}\right].
\end{equation}
Going back to (\ref{E:SNz}) we conclude that
\begin{equation}\label{E:limSNz}
\lim_{N\to\infty}S_Nz = \sum_{i=1}^m\la z,u_i\ra u_i,
\end{equation}
and this is just the orthogonal projection of $z$ onto the subspace   $S(H)$ spanned by $u_1,\ldots, u_m$.  \end{proof}

We also make the following observation, used in the proof of the preceding Lemma.

 \begin{lemma}\label{L:indep}
Suppose $u_1,\ldots, u_m$ are linearly independent in a Hilbert space $H$ and $Z_1\subset Z_2\subset\ldots$ are closed subspaces of $H$ whose union is dense in $H$.  Let $P_{Z_N}$ be the orthogonal projection onto $Z_N$. Then, for $N$ large enough, $P_{Z_N}u_1,\ldots, P_{Z_N}u_m$ are linearly independent vectors in $Z_N$.
\end{lemma}
We present the proof as a broader argument below. 

If $u_1,\ldots, u_m$  are linearly independent in a  finite dimensional vector space $E$ then $u'_1,\ldots, u'_m$ are also linearly independent when $u'_a$ is close enough to $u_a$ for each $a$; this follows, for example, by expressing linear independence of vectors as a determinant being non-zero. This result also holds if $E$ is any normed linear space. To see this let $C$ be the   infimum of $\norm{\sum_{a=1}^mr_au_a}$ with $r=(r_1,\ldots, r_m)$ running over the unit ``diamond'' (compact) in $\mbr^m$, which consists of all $r$ for which
$$|r_1|+\ldots +|r_m|=1.$$
Then $C$ is greater than zero because the infimum is actually realized at some point $r^* $   and the vectors $u_1,\ldots, u_m$ are linearly independent. Now let $u'_1,\ldots, u'_m\in E$ be such that 
$$\max_{a}\norm{u'_a-u_a}<C/2.$$
 If $\sum_{a=1}^m\lambda_au'_a=0$ then
\begin{equation}
\begin{split}
\norm{\sum_a\lambda_au_a} &=\norm{\sum_a\lambda_a(u_a-u'_a)} \\
&\leq \frac{C}{2}\sum_a|\lambda_a|.\end{split}
\end{equation}
On the other hand we have $\norm{\sum_a\lambda_au_a} \geq C\sum_a|\lambda_a|$, and so $\sum_a|\lambda_a|$ must be $0$, which means that each $\lambda_a$ is $0$.

\begin{prop}\label{P:limHZN2}
Let $H$ be a   Hilbert space and $Z_1\subset Z_2\subset \ldots$ a sequence of closed subspaces such that  $\cup_{N\geq 1}Z_N$ is dense in $H$. Suppose
$$R:H\to Y$$
is a continuous linear surjection onto a finite-dimensional vector space $Y$. Then $R|Z_N$ is surjective onto $Y$ for large $N$.
\end{prop}

\begin{proof} Since $R(Z_1)\subset R(Z_2)\subset\ldots$ is an increasing sequences of subspaces of the finite-dimensional space $Y$, the subspaces stabilize, in the sense that  $R(Z_{N_0})=R(Z_N)$ for all $N\geq N_0$, where $N_0$ is such that
$$\dim R(Z_{N_0})=\max_{N\geq 1}\dim R(Z_N).$$
   Let $y^*\in Y$ be orthogonal to $R(Z_{N_0})$; here we have equipped $Y$ with an arbitrary inner product. Let $J_N:Z_N\to H$ be the inclusion map. Then $J_N^*:H\to Z_N$ is the orthogonal projection onto the subspace $Z_N$.  Since, for all $N\geq N_0$, the vector $y^*\in Y$ is orthogonal to $R(Z_N)={\rm Im}(RJ_N)$ then $y^*\in \ker (RJ_N)^*=[{\rm Im}(RJ_N)]^\perp$. Thus,
    $$J_N^*R^*y^*=0.$$
    This means
     $$P_{Z_N}(R^*y^*)=0,$$
      since $J_N^*z=P_{Z_N}z$ for all $z\in H$ and all $N\geq N_0$. Then, using Lemma \ref{L:incsub},
$$R^*y^*=\lim_{N\to\infty}P_{Z_N}(R^*y^*)=0,$$
and so $y^*\in\ker R^*={\rm Im}(R)^\perp=0$.    Thus $R(Z_{N_0})=Y$, and so $R|Z_N$ is surjective for all $N\geq N_0$.\end{proof}

We apply Theorem \ref{T:limQN}  and  Proposition \ref{P:limHZN2} to obtain the following result.

\begin{theorem}\label{T:limHZN3}
Let $H$ be a   Hilbert space, and   $Z_1\subset Z_2\subset \ldots$ be a sequence of finite dimensional subspaces such that  $\cup_{N\geq 1}Z_N$ is dense in $H$.   Let ${\cl}:H\to X$ and $Q:H\to W$ be continuous linear surjections onto finite-dimensional Hilbert spaces $X$ and $W$, such that 
$$\ker {\cl}+\ker Q=H.$$
  Let ${\cl}_N={\cl}|Z_N$  and $Q_N=Q|Z_N$. Then 
  \begin{equation}\label{E:ZNkerPQN1}
Q_N(\ker {\cl}_N)=W\qquad\hbox{for large $N$,}
\end{equation}
and
  \begin{equation}\label{E:PNQN}
  \ker {\cl}_N+ \ker Q_N =Z_N\quad\hbox{ for large $N$.}
  \end{equation}
\end{theorem}
\begin{proof}    
By the observation (\ref{E:L0ZN}) made after Theorem    \ref{T:limQN}, the union of the finite-dimensional subspaces $Z_N\cap\ker {\cl}$  is dense in $\ker {\cl}$.  Considering now the operator $Q|\ker {\cl}$, and applying Proposition \ref{P:limHZN2}, we conclude that $Q|(Z_N\cap \ker {\cl})$ is surjective onto $W$ for $N$ large enough. Thus, since 
\begin{equation}\label{E:ZNkerPPN}
Z_N\cap\ker {\cl}=\ker {\cl}_N,
\end{equation}
we have $Q_N$ maps $\ker {\cl}_N$ onto $W$ for large $N$.
Then by Lemma \ref{L:transverse} it follows that $\ker Q_N+\ker {\cl}_N$ is $Z_N$ for such $N$.
\end{proof}


{\bf{ Acknowledgments.} }   This  research has been supported by     NSA grants   H98230-15-1-0254 and  H98230-16-1-0330.   Our thanks to Irfan Alam for many discussions on the subject. We are thankful to the online TikZ community and to Arthur Parzygnat for ideas for the figures. We are also very thankful to the referee for comments that have improved this paper.

\bibliographystyle{plain}

\end{document}